\documentclass[12pt, a4paper, reqno, twoside, makeidx]{amsart}
\usepackage[margin=3cm, marginpar=3cm]{geometry}
\usepackage{master}
\usepackage{verbatim}
\usepackage{overlay}

\usepackage{tikz}
\usetikzlibrary{arrows}
\tikzset{aar/.style={->, thick}}
\tikzset{taar/.style={double, double equal sign distance, -implies}}
\tikzset{amar/.style={->, dashed}}
\tikzset{amarp/.style={->, dashed,color=red}}
\tikzset{dmar/.style={->, dashed}}

\usepackage{amssymb}
\newcommand*{\QEDB}{\hfill\ensuremath{\square}}%

\newcommand{\Z}{\mathbb{Z}}

\newcommand{\R}{\mathbb{R}}
\newcommand{\C}{\mathbb{C}}

\newcommand{\Tt}{\mathbb{T}}
\newcommand{\F}{\mathbb{Z}_2}
\newcommand{\x}{\mathbf{x}}
\newcommand{\y}{\mathbf{y}}
\newcommand{\boa}{\boldsymbol{\alpha}}
\newcommand{\bob}{\boldsymbol{\beta}}

\newcommand{\T}{\mathbb{T}}

\begin{document}
\title[Knot Floer homology and curved bordered algebras]
{An introduction to knot Floer homology and curved bordered algebras}
\author{Antonio Alfieri}
\address{Central European University}
\email{alfieri\_antonio@phd.ceu.edu}
\address{University of California, Berkeley}
\email{jacksontvandyke@berkeley.edu}
\author{Jackson Van Dyke}
\maketitle

\vspace{-0.3cm}
\begin{abstract} We survey  Ozsv\' ath-Szab\' o's bordered approach to knot Floer homology. After a quick introduction to knot Floer homology, we introduce the relevant algebraic concepts ($\mathcal{A}_\infty$-modules, type  $D$-structures, box tensor product, etc.), we discuss partial Kauffman states, the construction of the boundary algebra, and sketch Ozsv\' ath and Szab\' o's analytic construction of the type $D$-structure associated to an upper diagram. Finally we give an explicit description of the structure maps of the $DA$-bimodules of some elementary partial diagrams. These can be used to perform explicit computations of the knot Floer differential of any knot in $S^3$. The boundary DGAs $\mathcal{B}(n,k)$ and $\mathcal{A}(n,k)$ of \cite{bordered1} are replaced here by an associative algebra $\mathcal{C}(n)$. These are the notes of two lecture series delivered by Peter Ozsv\' ath and Zolt\' an Szab\' o at Princeton University during the summer of 2018. 
\end{abstract}

\tableofcontents
\thispagestyle{empty}

\section{Introduction}
A powerful paradigm in computer science is the so called \textit{Divide et Impera}. It consists of dividing a complicated problem into smaller subproblems in order to solve them individually. In topology this translates to the \textit{cut and paste} philosophy. More specifically: suppose that one wants to compute an invariant $D$ of some space $X$, then one can divide $X$ into  pieces $P_1, \dots , P_N$ and can compute some sort of related invariant $I$ for the pieces, and develop some gluing formula expressing $D(X)$ as a function of $I(P_1), \dots , I(P_N)$. A typical example is when $D(X)$ is the fundamental group: in this case the invariants $I(P_i)$ are some inclusion maps, and the gluing principle is the Seifert-van Kampen Theorem. 

In four-manifold topology, cut and paste techniques to compute Seiberg-Witten invariants were pioneered in \cite{Froyshov} and eventually led to the definition of the Monopole Floer homology groups by Kronheimer and Mrowka \cite{kronheimer_mrowka_2007}. 
In the case of the Heegaard Floer three-manifold groups an implementation of the cut and paste approach first appeared in \cite{LOT}. In this work Lipshitz, Ozsv\' ath and Thurston associate to three manifolds with boundary $Y_1$ and $Y_2$ a type $D$-structure $CFD(Y_1)$  and a $\cA_\infty$-module $CFA(Y_2)$ respectively, so that $CF(Y_1 \cup_\partial -Y_2)=CFD(Y_1)\boxtimes CFA(Y_2)$ where, whatever these algebraic structures are, $\boxtimes$ is some sort of tensor product operation that takes as input a type $D$-structure and an $\cA_\infty$-module, and gives back a chain complex. 

Ozsv\' ath and Szab\' o developed a similar construction in the setting of knot Floer homology \cite{bordered1}. Given a knot $K \subset \R^3$ one can split it into two parts by means of a two-plane $z=t$. Denote by $K_{[t, +\infty)}=K \cap \{z\geq t\}$  the portion of $K$ lying above the plane $z=t$. Similarly denote  by $K_{(-\infty,t]}=K \cap \{z\leq t\}$ the part lying below it. In \cite{bordered1} Ozsv\' ath and Szab\' o associate to $K_{[ t, + \infty)}$ and $K_{( -\infty, t]}$ a type $D$-structure $DFK(K_{[t,+\infty)})$ and an $\cA_\infty$-module $AFK(K_{( -\infty, t])})$ so that 
\[CFK(K)= AFK(K_{( -\infty, t])}) \boxtimes DFK(K_{[t,+\infty)}) \ . \] 
More generally, they develop some sort of Morse theoretic approach one can apply to compute knot Floer homology: given $t_1<t_2 $ they associate to $K_{[t_1, t_2]}=K \cap \{t_1\leq z\leq t_2\}$ a type $DA$-structure\footnote{A type $DA$-structure is an algebraic structure that can be thought simultaneously as a type $D$-structure and an $\cA_\infty$-module. Type $D$- and $DA$-structures can be paired in order to produce a new type $D$-structure.} $DAFK(K_{[t_1, t_2]})$ such that 
\[DFK(K_{[t_1 , +\infty )})=DAFK(K_{[t_1, t_2]}) \boxtimes DFK(K_{[ t_2, +\infty)})  \ .\]
The scope of these lecture notes is to survey the algebraic language of $\cA_\infty$-modules, type $D$- and $DA$-structures, and sketch Ozsv\' ath and Szab\' o's recent construction. We will only assume some familiarity with Floer homology and the Fukaya category as in \cite{auroux_fukaya} (for motivational reasons), and with Heegaard Floer homology as in \cite{OS_intro}. 

\vspace{0.3cm}
\begin{small}
\textbf{Acknowledgements.}  The authors would like to thank Peter Ozsv\' ath, Zolt\' an Szab\' o and  Andr\' as Stipsicz, for many useful conversations. The first author is partially supported by the NKFIH Grant \' Elvonal (Frontier) KKP 126683 and from K112735.
\end{small}

\section{Knot Floer homology}
Knot Floer homology is a knot invariant introduced by Ozsv\' ath and Szab\' o 
\cite{OS_knots}, and independently by Rasmussen \cite{rasmussen_phd}.  
We presently review their construction and some related variations.

\subsection{Knot projections and Heegaard diagrams}\label{heegaarddiagram}
A Heegaard diagram $\left( \Sigma , \boa , \bob \right)$ of $S^3$
together with two base points $w$ and $z \in\Sigma$ lying in the complement of the $\alpha$- and the $\beta$-curves specifies an oriented knot $K\subset S^3$ as follows. 
Connect $w$ and $z$ in $\Sigma$ by an arc disjoint from the $\alpha$-curves, 
and push the resulting arc $a$ in the handlebody of the $\beta$-curves. 
Similarly, connect $w$ and $z$ by an arc disjoint from the $\beta$-curves, 
and push the resulting arc $b$ in the handlebody of the $\alpha$-curves. 
This produces a closed loop $K=a \cup b$ that can be oriented by taking $\p a = z - w = - \p b$.

\begin{figure}[t]
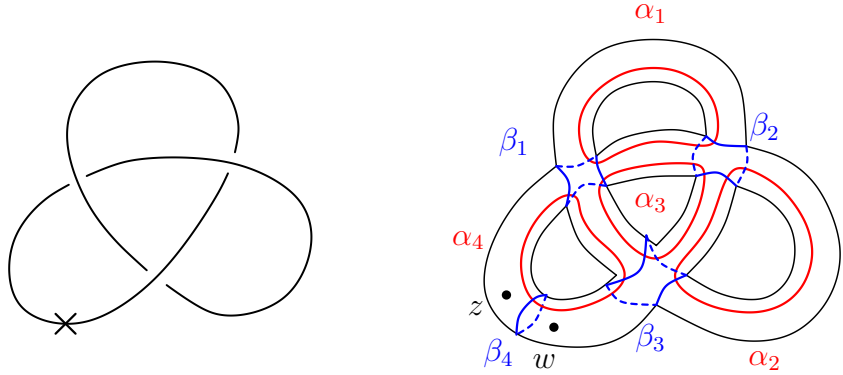

\begin{overlay}
\pict{doubly_pointed.pdf}{0.8}
\toptext{\textcolor{red}{$\al_1$}}{3,2.5}
\toptext{\textcolor{red}{$\al_2$}}{4.5,-2.1}
\toptext{\textcolor{red}{$\al_3$}}{3,0}
\toptext{\textcolor{red}{$\al_4$}}{0.6,-0.5}
\toptext{\textcolor{blue}{$\b_1$}}{1.2,0.8}
\toptext{\textcolor{blue}{$\b_2$}}{4.5,1}
\toptext{\textcolor{blue}{$\b_3$}}{3,-1.8}
\toptext{\textcolor{blue}{$\b_4$}}{1,-2}
\toptext{$z$}{0.7,-1.4}
\toptext{$w$}{1.6,-2.1}
\end{overlay}
\caption{Doubly-pointed Heegaard diagram of the trefoil.}
\label{fig:doubly_pointed}
\end{figure}

One can also go in the opposite direction. Consider a knot $K\subset S^3$. Let $V$  be a planar projection of $K$. This is a connected, $4$-valent planar graph. Thickening $V \subset \R^2\times 0 \subset \R^3 \cup \{ \infty \}$ we get a genus $g$ solid handlebody $U \subset S^3$ providing, together with its complement $U^c=\overline{S^3-V}$, a Heegaard splitting of $S^3$. As $\al$ curves of the corresponding Heegaard diagram (supported on $\Sigma = \p U$) we choose the boundary of the compact complementary regions of $V$. (This provides a set of compressing circles for $U_\alpha=U^c$). As compressing circles of $U_\beta=U$ ($\beta$-curves) we take one curve in correspondence with each crossing as suggested by Figure \ref{fig:doubly_pointed}, and one further $\beta$-curve representing a meridian of $K$. By taking base points on the two sides of the $\beta$-curve corresponding to the meridian we get a doubly pointed Heegaard diagram $\left( \Sigma , \boa , \bob , z , w \right)$ representing $K \subset S^3$. 

\begin{exr} Prove that in the diagram  associated to a knot projection there are as many $\alpha$-curves as $\beta$-curves, and that this number equals the genus.  
\end{exr}

\subsection{The Floer setup of a Heegaard diagram}
In \cite{OS_manifolds} Ozsv\' ath and Szab\' o associate to a Heegaard diagram $(\Sigma, \mb{\alpha}, \mb{\beta})$ a symplectic manifold $(M, \omega)$ together with two Lagrangians $L_{\mb{\alpha}}$ and $L_{\mb{\beta}} \subset M$. Their construction goes as follows. 

Endow $\Sigma$ with the structure of a Riemann surface and consider the space of degree $g$ divisors over $\Sigma$. This is the $g$-fold symmetric product $\Sym^g(\Sigma)$ of $\Sigma$
\begin{equation*}
\Sym^g( \Sigma ) = \Sigma^{\times g}/\mathfrak{S}_g\ .
\end{equation*}
Here $\mathfrak{S}_g$ denotes the symmetric group on $g$ letters acting on the $g$-fold cartesian product $\Sigma^{\times g}$  by permutation of the coordinates. Although $\mathfrak{S}_g$ does not act freely, $\Sym^g(\Sigma)$ is a smooth complex manifold. Notice that the projection $\pi: \Sigma^{\times g} \to \Sym^g(\Sigma)$ is an analytic covering with branching locus the \textit{fat diagonal} 
\[\Delta=\{  x_1+ \dots + x_g \in  \Sym^g(\Sigma) \ | \ x_i=x_j \text{ for some } i \not=j \} \ .\]

\begin{rmk} At first glance it could be surprising that $\Sym^g(\Sigma)$ is a smooth manifold. This is clear if we look at the local picture. Consider the $g$-fold symmetric product of the complex plane $\Sym^g(\C)$. This is the space of unordered $g$-tuples  of complex numbers $\z=\{z_1, \dots, z_g\}$ where each number can appear more than once. Sending an unordered $g$-tuple $\z=\{z_1, \dots, z_g\}$ to the coefficients $(a_0, \dots , a_{g-1})$ of the (monic) polynomial $\prod_i (x-z_i)=x^g+a_{g-1} x^{g-1}+ \dots +a_0 $ we get a map $\Sym^g(\C)\to \C^g$. It is easy to verify that, as a consequence of the Fundamental Theorem of Algebra, the map $\z \mapsto (a_0, \dots , a_{g-1})$ gives a homeomorphism $\Sym^g(\C)\simeq \C^g$. 
\end{rmk}

One can turn $\Sym^g(\Sigma)$ into a symplectic manifold as follows. Pick an area form $\nu \in \Omega^2(T\Sigma)$ taming the complex structure of $\Sigma$. This induces a symplectic form $\nu^{\times g}= \nu \times \dots \times \nu$ over $\Sigma^{\times g}$ taming the product complex structure of $\Sigma^{\times g}$. Now one would like to take the push forward $\omega= \pi_*(\nu^{\times g})$ of $\nu^{\times g}$ through the covering projection $\pi: \Sigma^{\times g} \to \Sym^g(\Sigma)$. Unfortunately, this can't be done since the form $\omega$ would be somehow singular near to the branching locus of $\pi$ (the fat diagonal). This problem can be overcame using the work of Perutz \cite{Perutz}. 

\begin{thm}\label{perutzform}
Suppose that $U$ is an open subset containing the fat diagonal $\Delta \subset \Sym^g(\Sigma)$. Then there exists a K\" ahler form $\omega$ on $\Sym^g(\Sigma)$ that agrees on $\Sym^g(\Sigma)\setminus \overline{U}$ with the push-forward $\pi_*(\nu^{\times g})$ of the product form. Furthermore, one can choose $\omega$ so that $[\omega]=\pi_*[\nu^{\times g}] \in H^2(\Sym^g(\Sigma); \R)$.
\end{thm}

From the attaching circles we can form two smoothly embedded $g$-dimensional tori $\TT_\al = \al_1 \times \cdots \times \al_g$ and $\TT_\b = \b_1 \times \cdots \times \b_g$ lying in the symmetric product $\Sym^g(\Sigma)$. Notice that $\TT_\al$ are $\TT_\b$ are Lagrangian submanifolds of $(\Sym^g(\Sigma), \omega)$.

\subsection{Intersection points and Kauffman states}
The intersection points $\TT_\al \cap \TT_\b$ of the $\alpha$- and $\beta$-tori of a doubly pointed Heegaard diagram $( \Sigma , \boa , \bob, z, w )$ associated to a knot projection can be interpreted combinatorially as follows.

Let $K$ be an oriented knot in $S^3$, and suppose that $V \subseteq \R^2$ is an oriented planar projection of $K$.
A \textit{domain} of $V$ is the closure of a connected component of its complement in $\R^2$. 
Let $Cr(V)$ denote the set of crossings in the projection $V$ and $D_0(V)$ the set of its domains. 
Choose a distinguished arc of the projection (in figures such an arc will be marked with the symbol $\times$) 
and denote by $D(V)$ the set of those domains which are disjoint from the marked arc. A \textit{Kauffman state} of the marked diagram $V$ is a bijection 
$\sigma : Cr(V) \to D(V)$ such that $c \in \sigma(c)$ for all $c \in Cr(V)$. 
Kauffman states are in obvious one to one correspondence with the intersection points of 
$\TT_\alpha\cap \TT_\beta$. 
A Kauffman state should be thought of as a choice of region for each crossing. 
In figures we will mark these regions with a black dot.
See \cref{fig:kauf_exm} for an example.

\begin{figure}
\includegraphics[width=0.2\textwidth]{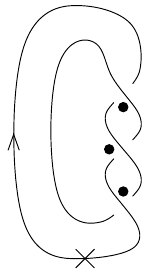}
\caption{One of three Kauffman states for the left-handed trefoil knot.}
\label{fig:kauf_exm}
\end{figure}

\subsection{The knot Floer chain complex}
Given a genus $g$, doubly pointed Heegaard diagram $(\Sigma, \mb{\alpha}, \mb{\beta}, z, w)$,
we can pick a symplectic form $\omega$ on $\Sym^g(\Sigma)$ as in Theorem \ref{perutzform}, a generic path $J_s\in \mathcal{J}(\Sym^g(\Sigma))$ of compatible almost-complex structures, and run the construction of Lagrangian Floer homology to get a differential on
\[ \widehat{CF}(\Sigma, \mb{\alpha}, \mb{\beta})= \bigoplus_{\x \in \TT_\alpha \cap \TT_\beta} \Z_2 \cdot \x  \ .\]
In doing so one would immediately run into some technical problems \cite{OS_manifolds}:
\begin{enumerate}
\item if $g> 2$ then $\pi_2(\Sym^g(\Sigma))=\Z$. A generator $S \in \pi_2(\Sym^g(\Sigma), \TT_\alpha)$  can be found as follows.  Take a hyperelliptic involution $\tau$ on $\Sigma$ and fix a basepoint $z \in \Sigma$, then $y \in \Sigma/\tau= S^2 \mapsto y+\tau(y)+(g-2)z \in \Sym^g(\Sigma) $ is a sphere representing $S$. 
\item Similarly, $\pi_2(\Sym^g(\Sigma), \TT_\alpha)= \pi_2(\Sym^g(\Sigma), \TT_\beta) = \Z$ for $g > 2$.
\end{enumerate} 
In particular both bubbling and boundary degenerations can occur. To fix this, one can choose a basepoint $z \in \Sigma\setminus \mb{\alpha} \cup \mb{\beta}$ and count holomorphic disks only in $\Sym^g(\Sigma- z)=\Sym^g(\Sigma)\setminus V_z$, where $V_z=z \times \Sym^{g-1}(\Sigma)=  \{ \x= x_1+ \dots+x_{g} \in \Sym^g(\Sigma) \ | \ x_i=z \text{ for some }i \}$. With this expedient we recover the conditions 
\begin{align*}
\pi_2(\Sym^g(\Sigma-z))=0 && \pi_2(\Sym^g(\Sigma-z), \TT_\alpha)= \pi_2(\Sym^g(\Sigma-z), \TT_\beta) = 0 \ , 
\end{align*}
and we can apply Floer's argument verbatim. We define the differential of $\widehat{CF}(\Sigma, \mb{\alpha}, \mb{\beta})$ to be
\begin{equation*}
\hatt{\p}\x= 
\sum_{\y \in \TT_\alpha \cap \TT_\beta}
\sum_{
\multunder{
\phi \in \pi_2\left( \x, \y \right) \\
n_z\left( \phi \right)= 0 \\
\mu\left( \phi \right) = 1 }}
\#\left( \frac{\cM\left( \phi \right)}{\RR} \right) \cdot \y \ ,
\end{equation*}
where $\pi_2\left( \x, \y \right)$ denotes the set of homotopy classes of disks connecting $\x$ to $\y$ (\text{i.e.} maps $u: D^2\simeq [0,1] \times\R \to \Sym^g(\Sigma)$  such that  $u \left( 0\times \R \right) \subseteq \T_{\boa}$, $u \left(  1 \times \R \right) \subseteq \T_{\bob}$, $\lim_{t\to - \infty } u(s+it)= \x$ and  $\lim_{t \to + \infty } u(s+it)= \y$), $\mu(\phi) \in \Z$ denotes  the expected dimension of the moduli space  $\cM\left( \phi \right)$ of $J_s$-holomorphic representatives in such a class, and $n_z(\phi)=\# \left( \phi^{-1}(V_z)\right)$  the algebraic intersection of a homotopy class with the divisor $V_z= z \times \Sym^{g-1}(\Sigma)$.

Notice that $(\widehat{CF}(\Sigma, \mb{\alpha}, \mb{\beta}), \partial)$ is independent from the choice of the basepoint $w$. In fact its homology does not depend on the knot \cite[Theorem 1.1]{OS_manifolds}: $H_*(\widehat{CF}(\Sigma, \mb{\alpha}, \mb{\beta}), \partial)= \Z_2$.  The choice of the second basepoint $w$ determines a filtration over $\widehat{CF}(\Sigma, \mb{\alpha}, \mb{\beta})$. This is the \textit{Alexander filtration} $A$ and it is characterized by the property 
\[A(\x) - A(\y)= n_w(\phi)- n_z(\phi) \]
where $\x$ and $\y \in \TT_\alpha \cap \TT_\beta$ denote two generators, and $\phi \in \pi_2(\x, \y)$. Looking at page $E_1$ of the associated spectral sequence we get a differential $\hat{\p}_K$ over $\widehat{CF}(\Sigma, \mb{\alpha}, \mb{\beta})$. One explicitly computes
\begin{equation*}
\hatt{\p}_K\x= 
\sum_{\y \in \TT_\alpha \cap \TT_\beta}
\sum_{\multunder{
\phi \in \pi_2\left( \x, \y \right) \\
n_w\left( \phi \right)= 0, \ n_z\left( \phi \right)= 0 \\
\mu\left( \phi \right) = 1}}
\#\left( \frac{\cM\left( \phi \right)}{\RR} \right) \cdot \y \ ,
\end{equation*}
\text{i.e.} we only count holomorphic disks that do not cross both the divisor $V_z$ and $V_w$. 

Another possibility is to count all $J_s$-holomorphic disks and record their intersection with both basepoints. Given a doubly pointed Heegaard diagram $(\Sigma, \boa, \bob, z, w)$ of a knot $K \subset S^3$, we can form the chain complex freely generated by its Heegaard states over the ring $\F[ U,V]/UV$
\begin{equation*}
CFK(K) \ceqq
\bdsum_{\x\in \TT_\al \cap \TT_\b}
\F[ U,V]/UV  \cdot \x \ .
\end{equation*}
This is equipped with the differential:
\begin{equation*}
\partial_K (\x)= \sum_{ \y \in \Tt_\alpha \cap \Tt_\beta}
\sum_{\multunder{\phi\in \pi_2\left( x,y \right) \\ \mu\left( \phi\right) = 1}}
\# \left( \frac{\cM\left( \phi \right)}{\RR} \right) U^{n_z\left( \phi \right)} V^{n_w\left( \phi \right)} \cdot 	\y \ .
\end{equation*}
In \cite{OS_manifolds} Ozsv\' ath and Szab\' o proved that the chain homotopy type of $CFK(K)$ does not depend on the choice of the doubly pointed  Heegaard diagram $(\Sigma, \boa, \bob, z, w)$ of $K$, nor on the generic path of almost complex structures $J_s \in \mathcal{J}(\Sym^g(\Sigma))$. We denote by $HFK(K)$ the homology of $CFK(K)$.

Imposing $U=0$ in the complex (instead of $UV=0$) we get a simplified version of the theory usually denoted by $HFK^-(K)$. Setting $U=V=0$ in $CFK(K)$ we get instead the homology of $\widehat{CFK}(K)=(\widehat{CF}(\Sigma, \mb{\alpha}, \mb{\beta}, z, w), \hat{\partial})$, usually denoted by $\widehat{HFK}(K)$.

\section{$\cA_\infty$-modules and type $D$-structures}
\subsection{$\cA_\infty$-algebras}
Let $R$ be a commutative ring with unit. We will often assume that $R$ has characteristic two, \text{i.e.} $1+1=0$.
An $\cA_\infty$-algebra $\cA$ is a graded $R$-module together with (grading-preserving) linear maps 
$\mu_{j }:\cA^{\tp j} \fromto \cA[2-j]$ defined for $j\geq 0$. These are required to satisfy the so called $\cA_\infty$-relations:  
\begin{equation*}
\sum_{i+j+\ell=n}\mu_{i+1+\ell} \circ (id_{\cA^{\otimes i}} \otimes \mu_j \otimes id_{\cA^{\otimes \ell}})=0 \ . 
\end{equation*}
Here $\cA\left[ 2-j \right]$ denotes the algebra $\cA$ with degree shifted by $2-j$. That is, degree $d+2-j$ elements of $\cA$ are declared to have grading $d$ in $\cA\left[ 2-j \right]$. (This guarantees that $\mu_1: \cA \to \cA$ drops the grading by one for example.) 

\begin{exr}
Write down the $\cA_\infty$-relations for $n=0,1,2$.  Prove that in an $\cA_\infty$-algebra $\partial \omega=0$ 
and $\partial^2a=\omega \cdot a + a \cdot \omega$, where we set $\omega=\mu_0\left( 1 \right)$, $\partial x= \mu_1\left( x \right)$, and $x\cdot y = \mu_2\left( x\tp y \right)$.
\end{exr}

$\cA_\infty$-relations can be understood graphically as follows. Consider trees which have $n+1$ leaves, a preferred leaf (called the root), and exactly two vertices.  Then each such tree $T$ represents a map $\mu_T :\cA^{\tp n} \to \cA$: the leaves of $T$ 
represent the inputs of $\mu_T$, the preferred leaf its output, and a vertex of valence $j$ represents the operation $\mu_j$. In this notation the $\cA_\infty$-relations say that for all $n \geq 1$
\begin{equation*}
\sum_T \mu_T(a_1 \otimes \dots \otimes a_n) = \sum_{i=1}^n \mu_{n+1}(a_1 \otimes \dots \otimes a_{i-1} \otimes  \omega \otimes a_{i}\otimes  \dots \otimes a_n )\ , 
\end{equation*}
where we set $\om = \mu_0\left( 1 \right)$ and
the sum on the left is extended to all such trees with 
$n+1$ leaves. An example for $n=4$ is shown in Figure \ref{fig:trees}. 

\begin{rmk} $\cA_\infty$-algebras generalize some classical algebraic structures. An $\cA_\infty$-algebra with $\mu_0=0$ and $\mu_j=0$ for $j \geq 2$ is a chain complex, a differential graded algebra (DGA) is an $\cA_\infty$-algebra with $\mu_0=0$ and $\mu_j=0$ for $j \geq 3$, while a graded associative algebra is simply an $\cA_\infty$-algebra with  $\mu_j=0$ for $j \not =2$.
\end{rmk}

If $\mu_0: R \to \cA$ is non-zero we say that $\cA$ is a \textit{curved} $\cA_\infty$-algebra. In this case, $\mu_0\left( 1 \right)=\omega$ is called the \textit{curvature} of $\mathcal{A}$. A curved $\cA_\infty$-algebra with $\mu_1=0$ and $\mu_j=0$ for $j \geq 3$  is just a graded algebra $\cA$ together with a preferred central element $\omega$, \text{i.e.} an element $\omega \in \cA$ such that $\omega \cdot x = x \cdot \omega$. We will call these objects curved associative algebras.

Another way to understand the $\cA_\infty$-relations is through the tensor algebra 
$T^*(\cA)= \bigoplus_{i=0}^\infty \mathcal{A}^{\otimes i}$. Notice that the maps $\mu_j$ ($j \geq 1$) sum up to a map $\mu:^\infty \mu_j: T^*(\cA) \to \mathcal{A}$. Define $\overline{D}: T^*(\cA) \to T^*(\cA)$ 
to be
\begin{equation}
\overline{D}(a_1 \otimes \dots \otimes a_n)= \sum_{j=1}^n \sum_{\ell=1}^{n-j+1} a_1 \otimes \dots \otimes \mu_j(a_\ell \otimes \dots \otimes a_{\ell+j-1})\otimes \dots \otimes a_n \ .
\label{eqn:bard}
\end{equation}
Then the $\cA_\infty$-relations can be written as $\mu \circ (\overline{D}+ i_\omega)=0$
where $\omega$ represents the curvature of $\cA$, and 
\begin{equation}
i_\omega(a_1 \otimes \dots  \otimes a_n)= \sum_{i=0}^{n+1} a_1 \otimes \dots \otimes a_{i-1} \otimes  \omega \otimes a_{i}\otimes  \dots \otimes a_n \ . 
\label{eqn:i}
\end{equation}
In what follows we will agree that $\overline{D}(1)=\mu(1)=0$ and $i_\omega(1)=\omega$, where $1\in R$ denotes the unit of the ground ring $R \subset T^*(\cA)=R\oplus \cA \oplus (\cA \otimes \cA) \oplus \dots$
 
\begin{figure}
\includegraphics[height=0.2\textheight]{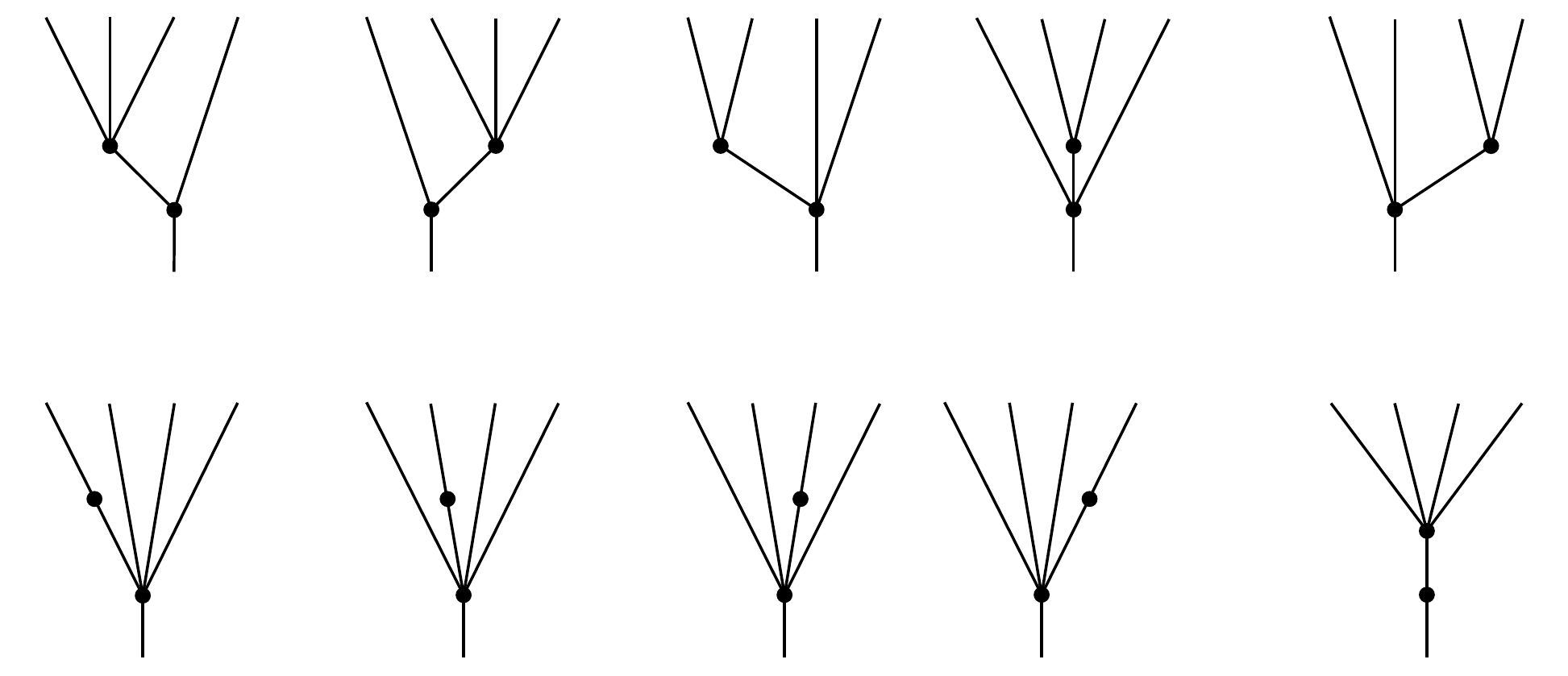}
\hfill
\includegraphics[height=0.2\textheight]{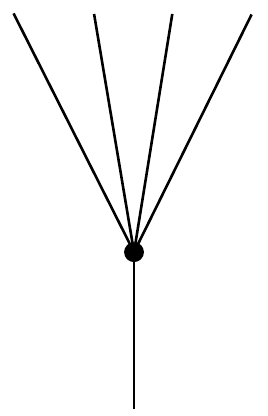}
\caption{Operational trees for $n=4$. 
The rightmost tree on the first line represents $\mu_3( a_1\otimes a_2\otimes \mu_2( a_3\otimes a_4 ) )$. Notice that by collapsing the edge in between the two nodes of any tree on the left we get the four-legged tree on the right. 
In other words all operational trees are obtained from the tree on the right by blowing-up the central node in all possible fashions.}

\label{fig:trees}
\end{figure} 

\subsection{$\cA_\infty$-modules} 
An $\cA_\infty$-module over an $\cA_\infty$-algebra $\cA$ is a graded, finitely generated $R$-module $M$ together with linear maps $m_{j }: M\tp \cA^{\tp j-1} \fromto M[2-i]$ defined for $j= 1, 2 , \dots$ such that
\begin{equation}
\label{eqn:a_infinity_module}
0 = \sum_{i + j = n+1} m_i \circ ( m_j \otimes id_{\cA^{\otimes i}}) + \sum_{i + j + k = n} m_{i  + 1+k}\circ (id_{\cA^{\otimes i}} \otimes \mu_j \otimes id_{\cA^{\otimes k}})
\end{equation}
for all $n \geq 1$. These are the $\cA_\infty$-relations for modules. As in the case of algebras, $\cA_\infty$-relations of modules can be interpreted graphically. Consider trees which have $n+1$ inputs and a single output, where the furthest left input is privileged. Then each such tree $T$ represents a map $m_T: M \otimes \cA^{\tp n} \fromto M$. Again $\cA_\infty$-relations say that for all $n \geq 0$
\begin{equation*}
\sum_T m_T(x \otimes a_1 \otimes \dots \otimes a_n) =\sum_{i=1}^n m_{n+1}( x \otimes a_1 \otimes \dots \otimes a_{i-1} \otimes  \omega \otimes a_{i}\otimes  \dots \otimes a_n ) 
\end{equation*}
where the sum is extended over all above mentioned trees. 

Alternatively, one can interpret the $\cA_\infty$-relations by means of the tensor algebra $T^*(\cA)$ as follows. Define  $m:T^*(\cA) \otimes M \to M$ by summing up the $m_j$ maps, and consider the canonical co-multiplication $\Delta : T^*(\cA) \to T^*(\cA) \otimes T^*(\cA)$
\[ \Delta(a_1 \otimes \dots \otimes a_n) = \sum_{i+j=n} (a_1 \otimes \dots a_i) \otimes (a_{i+1} \otimes \dots \otimes a_n) . \]
Then the $\cA_\infty$-relations can be graphically interpreted as insisting that
\[
\begin{tikzpicture}[baseline=(center.base)]
  \node at (0,-2) (center) {};
  \node at (0,0) (tc) {};
  \node at (2,0) (tr) {};
  \node at (1,-1) (Delta) {$\Delta$};
  \node at (0,-2) (ma) {$m$};
  \node at (0,-3) (mb) {$m$};
  \node at (0,-4) (bc) {};
  \draw[taar] (tr) to (Delta);
  \draw[taar] (Delta) to (ma);
  \draw[taar] (Delta) to (mb);
  \draw[amar] (tc) to (ma);
  \draw[amar] (ma) to (mb);
  \draw[amar] (mb) to (bc);
\end{tikzpicture}
\ \ \ + \ \ \
\begin{tikzpicture}[baseline=(center.base)]
  \node at (0,-1.5) (center) {};
  \node at (0,0) (tc) {};
  \node at (2,0) (tr) {};
  \node at (1,-1) (D) {$\overline{D}$};
  \node at (0,-2) (m) {$m$};
  \node at (0,-3) (bc) {};
  \draw[taar] (tr) to (D);
  \draw[taar] (D) to (m);
  \draw[amar] (tc) to (m);
  \draw[amar] (m) to (bc);
\end{tikzpicture}
\ \ \ + \ \ \
\begin{tikzpicture}[baseline=(center.base)]
  \node at (0,-1.5) (center) {};
  \node at (0,0) (tc) {};
  \node at (2,0) (tr) {};
  \node at (1,-1) (D) {$i_\omega$};
  \node at (0,-2) (m) {$m$};
  \node at (0,-3) (bc) {};
  \draw[taar] (tr) to (D);
  \draw[taar] (D) to (m);
  \draw[amar] (tc) to (m);
  \draw[amar] (m) to (bc);
\end{tikzpicture}
\ \ \ =0
\]
where dashed arrows represent module elements, 
and double arrows indicate elements of the tensor algebra $T^*(\cA)$. 
Recall $\bar D$ and $i_\om$ are defined in \eqref{eqn:bard} and \eqref{eqn:i}.
  
\subsection{Type $D$-structures} 
Let $\cA$ be an $\cA_\infty$-algebra with curvature $\omega$, and $X$ a graded, finitely generated $R$-module. Given a linear map $\delta: X\fromto \cA\left[ 1 \right] \tp X$ we can form maps $\delta^j: X \to \cA\left[ 1 \right]^{\otimes j} \tp X$ for $j=1, 2, \dots $ by simply iterating $\delta$ on the rightmost argument. We say that $(X, \delta)$ forms a ($\omega$-curved) type $D$-structure provided that 
the structure equation is satisfied:
\begin{equation}
\sum_{j=1}^\infty (\mu_j \otimes \id_X) \circ \delta^j = \omega \otimes \id_X \ . 
\label{eqn:structure_eqn_D}
\end{equation}
In terms of the structure map $\mu = T^*(\cA) \to \cA$ of the $\cA_\infty$-algebra $\cA$ this is the same as
requiring $(\mu \otimes id_X) \circ \overline{\delta}= \omega \otimes id_X$, where $\overline{\delta}:X \to  T^*(\cA) \otimes X$ is defined by $\overline{\delta}( x)=\sum_{j=0}^\infty \delta^j(x)$.
In what follows we will be mainly interested in $D$-structures over curved associative algebras. In this case the structure equation simplifies to 
\[(\mu_2 \otimes \id_X)  \circ (\id_\cA \otimes \delta)  \circ \delta=\omega \otimes \id_X  \ .\]

\begin{lem}\label{explicitcomputation}Let $(X, \delta)$ be a type $D$-structure, then
\begin{align*}
\begin{tikzpicture}[baseline=(center.base)]
  \node at (0,0) (tc) {};
  \node at (0,-1) (delta) {$\overline{\delta}$};
  \node at (-1,-2) (Delta) {$\Delta$};
  \node at (-2,-3) (bll) {};
  \node at (-1,-3) (bl) {};
  \node at (0,-3) (bc) {};
  \node at (0,-1.5) (center) {};
  \draw[dmar] (tc) to (delta);
  \draw[dmar] (delta) to (bc);
  \draw[taar] (delta) to (Delta);
  \draw[taar] (Delta) to (bll);
  \draw[taar] (Delta) to (bl);
\end{tikzpicture}
&& = &&
\begin{tikzpicture}[baseline=(center.base)]
  \node at (0,0) (tc) {};
  \node at (0,-1) (deltaa) {$\overline{\delta}$};
  \node at (0,-2) (deltab) {$\overline{\delta}$};
  \node at (-2,-3) (bll) {};
  \node at (-1,-3) (bl) {};
  \node at (0,-3) (bc) {};
  \node at (0,-1.5) (center) {};
  \draw[dmar] (tc) to (deltaa);
  \draw[dmar] (deltaa) to (deltab);
  \draw[dmar] (deltab) to (bc);
  \draw[taar] (deltaa) to (bll);
  \draw[taar] (deltab) to (bl);
\end{tikzpicture}
&& \text{and} &&
\begin{tikzpicture}[baseline=(center.base)]
  \node at (0,0) (tc) {};
  \node at (0,-1) (delta) {$\overline{\delta}$};
  \node at (-1,-2) (D) {$\overline{D}$};
  \node at (-2,-3) (bl) {};
  \node at (0,-3) (bc) {};
  \node at (0,-1.5) (center) {};
  \draw[dmar] (tc) to (delta);
  \draw[dmar] (delta) to (bc);
  \draw[taar] (delta) to (D);
  \draw[taar] (D) to (bl);
\end{tikzpicture}
&& = &&
\begin{tikzpicture}[baseline=(center.base)]
  \node at (0,0) (tc) {};
  \node at (0,-1) (delta) {$\overline{\delta}$};
  \node at (-1,-2) (D) {$i_\omega$};
  \node at (-2,-3) (bl) {};
  \node at (0,-3) (bc) {};
  \node at (0,-1.5) (center) {};
  \draw[dmar] (tc) to (delta);
  \draw[dmar] (delta) to (bc);
  \draw[taar] (delta) to (D);
  \draw[taar] (D) to (bl);
\end{tikzpicture} \ \ .
\end{align*}
\label{lem:d_structure_lemma}
\end{lem}
\begin{proof} The first identity is a consequence of the fact that $\delta^{i+j}=(id_{\cA^{\otimes j}} \otimes \delta^i) \circ \delta^j$. The second identity is in fact equivalent to the structure equation. We give a proof in the case of curved associative algebras just to provide an example of direct computation.

Let $\x_1, \dots , \x_m$ be a basis of $X$. Write:
\[\delta (\x_i)= \sum_{j=1}^m \xi_{ij} \otimes \x_j  \]
for some constants $\xi_{ij} \in \cA$. Notice that in linear coordinates the structure equation $(\mu_2 \otimes \id_X)  \circ (\id_\cA \otimes \delta)  \circ \delta=\omega \otimes \id_X$ translates to 
\begin{equation}
\sum_{\ell=1}^m \xi_{i\ell} \cdot \xi_{\ell j}= \begin{cases} \ \omega \ \ \ \text{ if } i=j \\ \  0 \ \ \ \text{ otherwise } \end{cases} \ .
\label{eqn:curvature}
\end{equation}
Using the coefficients $\xi_{ij}$ we obtain the following expression for $\overline{\delta}(\x_i)$
\[\overline{\delta}(\x_i)= 1 \otimes \x_i
+ \sum_{j}\xi_{ij} \otimes \x_j
+ \sum_{j, \ell}^m \xi_{ij} \otimes \xi_{j \ell } \otimes \x_\ell 
+ \sum_{j, \ell , k}^m \xi_{ij} \otimes \xi_{j \ell } \otimes \xi_{\ell k} \otimes \x_k + \dots
\]
Thus, applying $\overline{D}$ gives us
\begin{align*}\overline{D}\otimes \id_X (\overline{\delta}(\x_i)) 
&= \sum_{j, \ell} \xi_{ij} \cdot \xi_{j \ell } \otimes  \x_\ell \\
&+ \sum_{j, \ell, k } \xi_{ij} \cdot \xi_{j \ell } \otimes \xi_{\ell k} \otimes \x_k + \sum_{j, \ell , k} \xi_{ij} \otimes \xi_{j \ell } \cdot \xi_{\ell k} \otimes \x_k + \dots\\
&= \sum_{\ell}\left( \sum_{j} \xi_{ij} \cdot \xi_{j \ell }\right) \otimes  \x_\ell \\
&+ \sum_{ \ell, k }\left( \sum_{j}  \xi_{ij} \cdot \xi_{j \ell }\right) \otimes \xi_{\ell k} \otimes \x_k 
+ \sum_{j, k} \xi_{ij} \otimes \left( \sum_{\ell} \xi_{j \ell } \cdot \xi_{\ell k}\right) \otimes \x_k + 
\cdots 
\end{align*}
Imposing the curvature equation \eqref{eqn:curvature}, we conclude that 
\begin{align*}\overline{D} \otimes \id_X \left( \overline{\delta}(\x_i) \right) 
&= \omega \otimes  \x_i + \sum_{k }\omega \otimes \xi_{i k} \otimes \x_k + \sum_{ k} \xi_{ij} \otimes \omega \otimes \x_k + \cdots\\
&= i_\omega \otimes id_X(1 \otimes \x_i) + i_\omega \otimes id_X(\delta(\x_i))+ \cdots
\end{align*}
and we are done.
\end{proof}

\subsection{Box tensor product}
There is a natural pairing between $\cA_\infty$-modules and type $D$-structures. 
Suppose that $X$ is a type $D$-structure, and $M$ an $\cA_\infty$-module 
over a curved associative algebra $\cA$. 
Then we can equip the tensor product $M \tp_R X$ with the endomorphism
\begin{equation*}
D = \sum_{j = 0}^\infty
\left( m_{j+1} \tp \id_X \right) \comp \left( \id_M \tp \dd^j \right) \ .
\end{equation*}
Notice that this might not be well-defined since the sum can be infinite. 
This is not the case if, for example, either $M$ or $X$ is \emph{operationally bounded}, 
\text{i.e.} $m_j \equiv 0$ or $\delta^j \equiv 0$ for large values of the indices.

\begin{thm}
$\left( M\tp_R X , D \right)$ is a chain complex,
\text{i.e.} $D^2=0$. 
\end{thm}
\begin{proof}
By the definition of $\overline{\delta}$ and $\Delta$, we can write:
\begin{equation*}
D^2 = 
\begin{cd}
\arrow[dashed]{dd}&
\arrow[dashed]{d}\\
& \barr{\dd} \arrow[dashed]{dd}\arrow[Rightarrow]{dl}\\
m \arrow[dashed]{dd}&\\
& \barr{\dd} \arrow[Rightarrow]{dl}\arrow[dashed]{dd}\\
m\arrow[dashed]{d}&\\
\ & \
\end{cd} 
\ \ \ = \ \ \
\begin{cd}
\arrow[dashed]{ddd} && \arrow[dashed]{d}\\
&& \barr{\dd} \arrow[Rightarrow]{dl}\arrow[dashed]{dddd}\\
& \Delta \arrow[Rightarrow]{dl}\arrow[Rightarrow]{ddl}& \\
m\arrow[dashed]{d}&&\\
m\arrow[dashed]{d}&&\\
\ && \
\end{cd}
\end{equation*}
Now according to the structure equation \eqref{eqn:a_infinity_module}
for the $\cA_\infty$ module $M$ we can write
\begin{equation*}
\begin{cd}
\arrow[dashed]{ddd} && \arrow[dashed]{d}\\
&& \barr{\dd} \arrow[Rightarrow]{dl}\arrow[dashed]{dddd}\\
& \Delta \arrow[Rightarrow]{dl}\arrow[Rightarrow]{ddl}& \\
m\arrow[dashed]{d}&&\\
m\arrow[dashed]{d}&&\\
\ && \
\end{cd}
\ \ \ =  \ \ \ 
\begin{cd}
\ \arrow[dashed]{ddd} & & \arrow[dashed]{d} \\
&& \barr{\dd} \arrow[Rightarrow]{dl}\arrow[dashed]{ddd}\\
& \barr{D} \arrow[Rightarrow]{dl}&\\
m\arrow[dashed]{d}&&\\
\ && \
\end{cd}
+
\begin{cd}
\ \arrow[dashed]{ddd} & & \arrow[dashed]{d} \\
&& \barr{\dd} \arrow[Rightarrow]{dl}\arrow[dashed]{ddd}\\
& i_\om \arrow[Rightarrow]{dl}&\\
m\arrow[dashed]{d}&&\\
\ && \
\end{cd}
= 0
\end{equation*}
where this last expression is $0$ by Lemma \ref{lem:d_structure_lemma}.
\end{proof}

We define  $N \boxtimes X =(N \otimes_R X, \delta)$ to be the box tensor product of $N$ and $X$.

\subsection{Flat type $D$-structures as chain complexes} A type $D$-structure $(X, \delta)$ 
over a curved associative algebra $(\cA, \omega)$ is called \textit{flat} if $\omega=0$. 
A flat type $D$-structure can be seen as a chain complex: extend the coefficients of $X$ by taking $X_\cA= \cA \otimes X$ and set $\partial= (\mu_2 \otimes id_X)\circ( id_\cA \otimes  \delta)$. If $X$ is free with basis $B$ then: 
\[X_\cA=\bigoplus_{\x \in B} \cA \cdot \x \ \ \ \ \ \text{ and } \ \ \ \ \ \partial(\x)=\sum_{\y \in B}   \xi_{\x, \y} \cdot \y \ , \]
where $\xi_{\x, \y}\in \cA$ denote the \textit{structure constants} of $\delta$, \text{i.e.}  
\[\delta(\x)=\sum_{\y\in B}   \xi_{\x, \y} \otimes \y \ . \] 
In what follows we will tacitly consider flat type $D$-structures as chain complexes in the sense specified above.  

\begin{rmk}
In general, if $\omega\not=0$ one has $\partial^2(x)=\omega \cdot x$. To fix this defect and get $\partial^2=0$ one can take $X_\cA= X \otimes \cA / I \cdot X \otimes \cA$ as a module over $\cA_I=\cA/I\cA$, where $I \subset \cA$ denotes the (principal) ideal generated by $\omega$.
\end{rmk}

\subsection{Curved $DA$-bimodules} Let $(\cA, \omega_{\text{new}})$ and $(\cB, \omega_{\text{old}})$ be two curved associative algebras possibly over two different rings $R_1$ and $R_2$. A type $DA$-bimodule $N$ is a graded bimodule over $R_1$ and $R_2$ (acting respectively on the left and the right) together with linear maps
\begin{equation*}
\dd_{1 + j}: N\tp_{R_2} \cB^{\tp j} \fromto \cA \tp_{R_1} N
\end{equation*}
% \todo{page 13 of bordered 1}
defined for $j \geq 0$. These are required to satisfy the structure equations:
\begin{equation*}
\begin{cd}
& \arrow[dashed]{dd}&\arrow[Rightarrow]{d}\\
&& \barr{D}\arrow[Rightarrow]{dl}\\
& \dd \arrow{dl}\arrow[dashed]{d}&\\
\ & \ &
\end{cd}
\ \ + \ \
\begin{tinycd}
& \arrow[dashed]{dd}&\arrow[Rightarrow]{d}\\
&& \Delta\arrow[Rightarrow]{dl}\arrow[Rightarrow]{ddl}\\
& \dd \arrow{ddl}\arrow[dashed]{d}&\\
& \dd \arrow{dl}\arrow[dashed]{dd}&\\
\mu_2\arrow{d} &  & \\
\ & \ &
\end{tinycd}
\ \ + \ \
\begin{cd}
&\arrow[dashed]{dd}&
\arrow[Rightarrow]{d}\\
&& i_{\om_{\text{old}}} \arrow[Rightarrow]{dl}\\
&\dd \arrow[dashed]{d}\arrow{dl}& \\
\ &\ &
\end{cd}
\ \ \ = \ \ \ 0
\end{equation*}
\begin{equation*}
\begin{cd}
& \arrow[dashed]{d}\\
& \dd_2 \circ \left( \id \tp \om_{\text{\text{old}}} \right)\arrow{dl}\arrow[dashed]{d} \\
\ & \
\end{cd}
\ + \
\begin{cd}
& \arrow[dashed]{d}\\
& \dd_1\arrow[dashed]{d}\arrow{ddl}\\
& \dd_1\arrow[dashed]{dd}\arrow{dl}\\
\mu_2 \arrow{d} & \\
\ & \
\end{cd}
\ \ \ = \ \ \ \om_{\text{new}} \tp \id
\end{equation*}
where $\delta=\sum_{j=1}^{\infty} \delta_{j}$. Here double arrows denote elements of the tensor algebra $T^*(\cB)$, dashed arrows elements of $N$, and plain arrows elements of $\cA$.

\subsection{Box tensor product and $DA$-bimodules}
Just as in the case of type $D$-structures and $\cA_\infty$-modules, one can pair type $D$-structures with $DA$-bimodules, and $DA$-bimodules with one another.

Let $\cA$ and $\cB$ be $\cA_\infty$-algebras, $(X, \delta_X)$  a $D$-structure over $\cB$, and $(N, \delta_{1}, \delta_2, \dots)$ a $DA$-bimodule over $\cA$ and $\cB$. On the tensor product $N \otimes X$ we define   
\begin{equation}
\delta= \sum_{j=0}^{+\infty} (\delta_{1+j} \otimes id_X) \circ (id_N \otimes \delta_X^j) \ .
\label{eqn:d_structure_map_on_tensor_product}
\end{equation}
This gives a map $\delta: N \otimes X \to \cA \otimes N \otimes X$. As in \Cref{lem:structure_map} one can check that $\delta$ is the structure map of a type $D$-structure on $N \otimes X$. We define the type $D$-structure $N \boxtimes X =(N \otimes X, \delta)$ to be the box tensor product of $N$ and $X$.

Similarly, we can pair $DA$-bimodules with one another. Let $(M, \delta_{1}^M, \delta_2^M, \dots)$ be a $DA$-bimodule over $\cA$ and $\cC$, and $(N, \delta_{1}^N, \delta_2^N, \dots)$ a $DA$-bimodule over $\cC$ and $\cB$. Then we can endow $M\otimes N$ with the $DA$-bimodule map $\delta^{M\otimes N}_{1+*}: N \tp T^*(\mathcal{B}) \to \mathcal{A}[1] \otimes N$ defined graphically by
\begin{equation*}
\begin{cd}
& \arrow[dashed]{dd}&\arrow[dashed]{d}&\arrow[Rightarrow]{dl}\\
&& \dd^N \arrow[Rightarrow]{dl}\arrow[dashed]{dd}& \\
& \dd^M \arrow{dl}\arrow[dashed]{d} && \\
\ & \ & \ &
\end{cd}
\end{equation*}
where $\dd^M= \sum_{j=1}^\infty \delta_j^M$ and $\dd^N= \sum_{j=1}^\infty \delta_j^N$. The resulting $DA$-bimodule (over $\cA$ and $\cB$) is the box tensor product of $M$ and $N$ and will also be denoted by $M \boxtimes N$.

\section{Partial Kauffman states}
\label{sec:partial_kauffman_states}
Let $K \subset \R^3$ be a knot. Denote by $\pi(x,y,z)=(x,z)$ the projection on the $xz$-plane and by $h(x,y,z)=z$ the height function. Suppose that $\restr{\pi}{K}$ is a regular projection and that $\restr{h}{K}$ is a Morse function with all local maxima at $h=+10$ and global minimum at $h=-10$. Fix $t\in ( -10 , 10 )$ and consider the plane $z=t$. This intersects $K$ transversely in $2n$ points and splits it in two parts: $K_{[t,+\infty)}=K \cap \{z\geq t\}$ the part lying above the plane, and $K_{( -\infty, t]}=K \cap \{z\leq t\}$ the part lying below it. Recall our scope is to associate to $K_{[t,+\infty)}$ a type $D$-structure $DFK(K_{[t,+\infty)})$ and to $K_{(-\infty,t]}$ an $\cA_\infty$-module $AFK(K_{(-\infty,t]})$ so that 
\[CFK(K)= AFK(K_{(-\infty,t]}) \boxtimes DFK(K_{[t,+\infty)})  \ . \]
In fact,  we associate to the portion of $K$ lying in between two parallel planes $K_{[t_1, t_2]}=K \cap \{t_1 \leq z\leq t_2\}$ a $DA$-bimodule $DAFK(K_{[t_1,t_2]})$ so that
\[DFK(K_{[t_1, +\infty)})= DAFK(K_{[t_1,t_2]}) \boxtimes DFK(K_{[t_2,+\infty)})  \ . \]
In this section we describe the generators of the vector spaces underlying $DFK$, $DAFK$ and $AFK$ in terms of combinatorial data associated to the projection.   

\subsection{Upper Kauffman states}
\label{sec:upperkauffman}
Let us start with $DFK(K_{[t,+\infty)})$. Consider the projection $D_{[t,+\infty)}= \pi\left( K \right) \cap \R \times \cop{t,+\infty}$  of the portion of $K$ lying above the plane $z=t$ on the plane $y=0$. This is the \textit{upper diagram} of the slice $z=t$. Choose $t \in (-10,10)$ so that there are no crossings or local minima at level $t$.

A \emph{local state}, is a subset of $n$ elements of the set $\left\{ 1 , \cdots , 2n-1 \right\}$. 
This corresponds to choosing $n$ of the $2n-1$ bounded intervals separated by the strands on the line $z = t$. 
In figures we will denote the choice of such spaces with a black bar. 
See \cref{fig:upper_exm} for an example.

\begin{figure}
\includegraphics[width=0.9\textwidth]{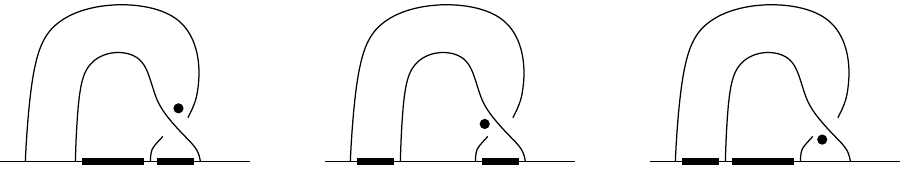}
\caption{The three upper Kauffman states for this upper diagram. 
Note that this is the upper diagram from the projection of the left-handed
trefoil of \Cref{fig:kauf_exm}.}
\label{fig:upper_exm}
\end{figure}

An \emph{upper Kauffman state} is a pair $\left( \k , x \right)$, where $x$ is an local state, and $\k$ is a function associating to each of the crossings in the upper diagram a connected component of $\R \times [t, +\infty) \setminus D_{[t,+\infty)}$. For an upper Kauffman   state we require the following conditions to be met:
\begin{enumerate}
\item for each crossing $c$ of $D_{[t,+\infty)}$ the region $\kappa(c)$ is one of the (at most four) regions having $c$ as boundary corner, 
\item $\kappa$ is injective and does not allow the unbounded region to be occupied,
\item the unbounded region meets none of the intervals in the local state $x$,
\item for every bounded, unoccupied region, there exists exactly one interval in $x$ which bounds the region.
\end{enumerate}
In figures the occupied regions will have a black dot.
Define $DFK(K_{[t,+\infty)})$ to be the vector space generated over $\F$ by the upper Kauffman states. 

\subsection{Partial Kauffman states}
Now instead of considering an upper diagram, consider a \emph{partial knot diagram}. This consists of $D_{[t_1, t_2]}= \pi(K) \cap \R \times [t_1, t_2]$ the portion of the front projection of $K$ between two horizontal planes $z=t_1$ and $z=t_2$ for $t_1<t_2$. 
A \emph{partial Kauffman state} is a triple $\left( \k , x , y \right)$ where $x$ is a local state of the $z=t_1$ slice, $y$ is a local state of the $z=t_2$ slice, and $\k$ is a map associating to each of the crossings of $D_{[t_1, t_2]}$ a bounded connected component of $\R \times [t_1, t_2] \setminus D_{[t_1, t_2]}$. For a partial Kauffman state the following conditions are required to hold:
\begin{enumerate}
\item $\kappa$ is injective, and $c \in \partial \kappa(c)$ for each crossing $c$ of $D_{[t_1, t_2]}$,
\item if a region $\cS$ is occupied by $\kappa$, then $y$ contains all the intervals at $z=t_2$ intersecting $\cS$, and $x$ contains none of the intervals at $z=t_1$ lying on $\cS$,
\item if $\cS$ is unoccupied by $\kappa$, then either:
\begin{enumerate}
\item  $y$ contains all but one of the intervals at $z=t_2$ intersecting $\cS$, and $x$ contains none of the  intervals at $z=t_1$ intersecting $\cS$
\item or $y$ contains all of the intervals at $z=t_2$ intersecting $\cS$, and $x$ contains exactly one of the intervals at $z=t_1$ intersecting $\cS$.
\end{enumerate}
\end{enumerate}
We denote by $DAFK(K_{[t_1,t_2]})$ the vector space generated over $\Z_2$ by the partial Kauffman states.
In figures the occupied regions will have a black dot.

\subsection{Lower Kauffman states and gluing} There is an obvious notion of gluing between upper and partial Kauffman states: 
given a Kauffman state $(\kappa_\cU, y)$ of $D_{[t_2, +\infty)}$, and a Kauffman state $(\kappa_\cP,  x, y)$ of $D_{[t_1, t_2]}$ 
we can form a Kauffman state  $(\kappa_\cP,  x, y)*(\kappa_\cU, y)=(\kappa_\cP * \kappa_\cU, x)$ of $D_{[t_1, +\infty)}$ by simply setting $\kappa_\cU * \kappa_\cP(c)=\kappa_\cU(c)$ if $c$ is a crossing of $D_{[t_1, +\infty)}$, and $\kappa_\cU * \kappa_\cP(c)=\kappa_\cP(c)$ otherwise. 

\begin{figure}[t]
\begin{overlay}
\pict{split_knot.pdf}{0.5}
\toptext{$\mathcal{U}$}{0,2}
\toptext{$\mathcal{L}$}{0,-0.8}
\toptext{$z = t$}{3.3,0.9}
\toptext{$z = -10+\e$}{3.3,-2}
\end{overlay}
\caption{Knot divided into the portion $\cU$ and $\cL$.}
\label{fig:split_knot}
\end{figure}

\begin{exr}
Choose $\epsilon >0$ so that $D_{[-10+\epsilon, t]}=\pi(K) \cap \R \times [-10+\epsilon, t]$ does not contain the global minimum as in Figure \ref{fig:split_knot}. Suppose that we have an upper Kauffman state $(\k_\cU, y)$ of  $D_{[t,+\infty)}$, and a partial Kauffman state $(\k_\cL, x, y)$ of the middle region $D_{[-10+\epsilon,t]}$.  Show that the gluing $\kappa_\cL*\kappa_\cU$ defines a Kauffman state of the knot at large. Show that all the Kauffman states of $K$ arise this way.
\end{exr}

Fix  $\epsilon >0$ sufficiently small so that $D_{[-10+\epsilon, t]}$ does not contain the global minimum, and define $AFK(D_{(-\infty, t]})$ to be the $\F$-vector space generated by the partial Kauffman states of $D_{[ -10+\epsilon, t]}$, \text{i.e.} set $AFK(D_{(-\infty, t]})=DAFK(D_{[-10+\epsilon, t]})$. The generators of $AFK(D_{(-\infty, t]})$ are called \textit{lower Kauffman states}.

\section{Boundary algebras, curvature and matchings}
\subsection{The boundary algebra $\cC\left( n \right)$ }
The type $D$-structures and  $\cA_\infty$-modules associated to upper and lower diagrams 
will be defined over a certain associative algebra $\cC\left( n \right)$
which only depends on the number of endpoints of the diagrams. 
The algebra $\cC\left( n \right)$ is the quotient of a larger algebra $\cC_0\left( n \right)$ 
which we will define presently.

As a vector space over $\Z_2$, $\cC_0\left( n \right)$ is generated by the set of triples:
\begin{equation*}
\left( x , y , U_1^{a_1} \cdots U_{2n}^{a_{2n}}\right) 
\end{equation*}
where $x$, $y$ are two local states, and $U_1^{a_1} \cdots U_{2n}^{a_{2n}}$ is a monomial in $\FF\left[ U_1 , \cdots , U_{2n} \right]$. 

Define the weight vector $w\left( x,y,U_1^{a_1} \cdots U_{2n}^{a_{2n}}\right) \in \Z^{2n}$ 
of a generator by setting
\begin{equation*}
w\left( x , y , U_1^{a_1} \cdots U_{2n}^{a_{2n}}\right)=
\frac{1}{2} v\left( x,y \right) + \left( a_1 , \cdots , a_{2n} \right)
\end{equation*}
where the $i$th component $v_i(x,y)$ of $v\left( x,y \right) \in \ZZ^{2n}$ is given by:
\begin{equation} v_i(x,y)=
\abs{
\# \left\{ j\in x \st i \leq j \right\} - 
\# \left\{ k\in y \st i \leq k \right\}
} \ .
\end{equation}
This can be thought of as a vector which keeps track of how $x$ changes into $y$.
See \Cref{fig:weight_2} for an example.
We define the product $a \cdot b$ of two generators $a = \left( x , y , U_1^{a_1} \cdots U_{2n}^{a_{2n}}\right)$ and $b = \left( s , r , U_1^{b_1} \cdots U_{2n}^{b_{2n}}\right)$ of $\cC_0\left( n \right)$ to be zero if $y \neq s$, and
\begin{equation*}
a \cdot b = \left( x,r , U_1^{t_1} \cdots U_{2n}^{t_{2n}} \right)
\end{equation*}
if $y = s$, where $t_1 , \cdots , t_{2n}$ are chosen so that
$w\left( a \cdot b \right) = w\left( a \right) + w\left( b \right)$.

For the local state $x$ set $I_x=\left( x , x, 1 \right)$. A straightforward computation shows that 
\begin{equation*}
I_x \cdot I_y = 
\begin{cases}
I_x & x = y \\
0 & x\neq y
\end{cases}\, \, \, .
\end{equation*}
In particular, $I_x$ is an idempotent element. In fact, $\left\{ I_x \st x \text{ is a local state } \right\}$ generate the ring of idempotents of $\cC_0\left( n \right)$.  This ring will be denoted by $I\left( n \right)$. Notice that $\cC_0\left( n \right)$ is  a unital algebra over $\FF\left[ U_1 , \cdots , U_{2n} \right]$, where
\[1 = \sum_{x\in \text{ local states }} I_x \ , \ \ \ \text{ and } \ \ \ U_i = \sum_{x\in \text{ local states }} ( x , x , U_i ) \ \ \ \text{for } i=1, \dots , 2n  .\]
The algebra $\cC\left( n \right)$ is defined by imposing on $\cC_0(n)$ the following relations:
\begin{enumerate}[label=(R-\arabic*)]
\item $\left( x,y,U_1^{t_1} \cdots U_{2n}^{t_{2n}} \right) = 0$ if $x\cap \left\{ i-1 , i \right\} = \emp$ and $t_i \neq 0$. \label{item:r1}
\item $\left( x,y,U_1^{t_1} \cdots U_{2n}^{t_{2n}} \right) = 0$ if the local states $x,y$ are far apart from one another, i.e. there is some $i$ such that $\abs{x_i - y_i} > 1$. \label{item:r2}
\end{enumerate}
The relations in \ref{item:r1} can be written in terms of indenpotent elements as 
\[I_x \cdot U_i=0\]
for any local state $x$ such that $x\cap \left\{ i-1 , i \right\} = \emp$. Similarly, the relations in \ref{item:r2} can be interpreted as follows. Fix $i \in \{1, \dots, 2n-1\}$. Given a local state $x$ such that $ i-1\in x$ but $i\not \in x$ we form a new local state $x'=(x\setminus \{j-1\}) \cup \{j\}$. Define  
\begin{equation*}
R_i \ceqq \sum_{\{x\st i-1\in x, i\not \in x\}} \left( x , x' , 1 \right)
\ \ \ \ \text{ and } \ \ \ \
L_i \ceqq 
\sum_{\{x\st i-1\in x, i\not \in x\}}
\left( x' , x , 1 \right).
\end{equation*}
Then to say that the relations in \ref{item:r2} are satisfied is the same as requiring that $R_i \cdot R_{i+1} = 0$ and $L_{i+1} \cdot L_i = 0$ for $i = 1, \dots , 2n-1$.

\begin{figure}[t]
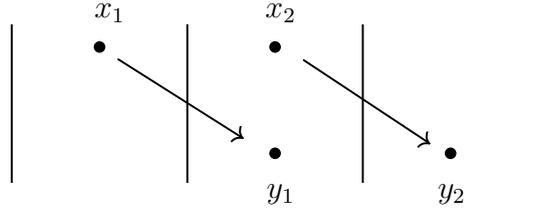

\centering
\begin{overlay}
\pict{weight_2.pdf}{0.5}
\toptext{$x_1$}{-2.25,1.2}
\toptext{$x_2$}{0,1.2}
\toptext{$y_1$}{0,-1.2}
\toptext{$y_2$}{2.25,-1.2}
\end{overlay}
\caption{
Take $n = 2$. In this case $v\left( x , y \right) \in \ZZ^{4}$.
In particular, 
$v\left( x,y \right) = \left(\abs{2-2},\abs{1-2},\abs{0-1} , \abs{0-0} \right) = \left( 0,1,1,0 \right)$.
This should be thought of as a vector which keeps track of how to `get'
from one local state to another by counting which boundaries are traversed.}
\label{fig:weight_2}
\end{figure}

\subsection{Action of idempotents on partial Kauffman states} 
Let $K_{[t, +\infty)}$ be an upper diagram with $2n=\abs{K \cap \{z=t\}}$ endpoints. 
Given an idempotent element $I_x \in I(n)$ and an upper Kauffman state $( \kappa , y)$,
set $I_x \cdot (\kappa , y)= 0$ if $x \not= y$, and $I_x \cdot (\kappa , x)= (\kappa , x)$ otherwise. This defines an action of the ring $I(n) \subset \cC(n)$ on the vector space $DFK(K_{[t, +\infty)})$ generated by upper Kauffman states.

Similarly, given a partial diagram $K_{[t_1, t_2]}$ we can consider the algebra $\cC(n_2)$ associated to the upper boundary of  $K_{[t_1, t_2]}$ (the slice at $z=t_2$, in particular $2n_2=\abs{K \cap \{z=t_2\}}$) and the algebra $\cC(n_1)$ associated to its lower boundary component ($2n_1=\abs{K \cap \{z=t_1\}}$). Given a partial Kauffman state $(\kappa, x , y)$ of $K_{[t_1, t_2]}$ and idempotent states $I_{z'} \in I(n_1)$ and $I_{z} \in I(n_2)$ we define 
\[
 (\kappa, x , y) \cdot I_z= 
\begin{cases}
(\kappa, x , y) & z = y \\
0 & z\neq y
\end{cases}
\ \ \  \text{ and }\ \ \
I_{z'} \cdot (\kappa, x , y)  = 
\begin{cases}
(\kappa, x , y) & z' = x \\
0 & z'\neq x
\end{cases} \ \ \ .
\] 
This turns the vector space generated by partial Kauffman states into a bimodule over $I(n_1)$ and $I(n_2)$. Notice that there are the following identifications of $\F$-vector spaces.

\begin{prop} $DFK(K_{[t_1, +\infty)})= DAFK(K_{[t_1, t_2]}) \otimes_{I(n_2)} DFK(K_{[t_2, +\infty)})$.
\end{prop}
\begin{proof}
Every upper Kauffman state of $K_{[t_1, +\infty)}$ is obtained by gluing an upper Kauffman state 
$(\kappa, y)$ of $K_{[t_2, +\infty)}$ to a partial Kauffman state 
$(\kappa',  y, x)$ of $K_{[t_1, t_2]}$. 
We make a generator $(\kappa', x, y) \otimes (\kappa, y)$ correspond to the gluing 
$(\kappa'*\kappa, x)$. 
For a generator of the form $(\kappa', x, y) \otimes (\kappa, z)$ with 
$z \not= y$ one computes 
$(\kappa', x, y) \otimes (\kappa, z) =  (\kappa', x, y) \cdot I_y \otimes  I_z \cdot (\kappa, z)  = (I_y \cdot I_z)   \cdot (\kappa', x, y) \otimes (\kappa, z)=0$ (where we have used the identity: $I_y \cdot I_z=0$ for  $z \not= y$), so we have nothing to map.
\end{proof}

\begin{cor}$\widehat{CFK}(K)= AFK(K_{(-\infty, t]}) \otimes_{I(n)} DFK(K_{[t, +\infty)})$.
\end{cor}

\subsection{Curvature and matchings} 
$DFK(K_{[t, +\infty)})$ and $AFK(K_{(-\infty,t]})$ will be respectively regarded as a curved type $D$-structure and $\cA_\infty$-module over $\cC(n)$, where $2n$ denotes the number of points in the slice $z=t$. The curvature has the following combinatorial interpretation.

For the slice of a knot at $z=t$, we have an associated matching of the $2n$ intersection
points to one another given by how they are connected by the strands lying above the plane $z=t$.
This is denoted $M$, and consists of $n$ disjoint pairs $\{ i,j \}$ for $i , j \in \left\{ 1 , \cdots , 2n \right\}$.
The curvature is given by:
\[\om_M = \sum_{ \{ i,j \} \in M} U_i U_j \ .\]

\begin{exr} Prove that $\omega_M$ defined from a matching $M$ of $\{1, \dots , 2n\}$ as above is a central element of $\cC(n)$. That is, prove that $\omega_M \cdot x= x  \cdot \omega_M$ for all $x \in \cC(n)$.
\end{exr} 

\section{Structure maps and the Gluing Theorem}
\subsection{The type $D$-structure of an upper diagram}\label{typeD}
We now sketch the definition of the structure map
\begin{equation*}
\dd : DFK(K_{[t, +\infty)}) \fromto \cC\left( n \right) \otimes_{I(n)} DFK(K_{[t, +\infty)}) \ ,
\end{equation*}
of the type $D$-structure associated to an upper diagram $K_{[t, +\infty)}$. Here the ring of indepotents $I(n)$ plays the role of the base ring, while $\cC(n)$ the one of the ground $\cA_\infty$-algebra.  The structure constants of $\delta$ will be defined by means of pseudo-holomorphic disk counts. First let us interpret upper Kauffman states in terms of intersection points of the appropriate Lagrangian Floer setting.

Let $K\subset \R^3$ be a knot. Put $K$ in standard position as in \Cref{sec:partial_kauffman_states} and form the doubly pointed Heegaard diagram $\cH$ associated to the projection of $K$ on the $xz$-plane as in Section \ref{heegaarddiagram}. Looking at the portion of $\cH$ lying above the plane $z=t$ for some $t \in (-10, 10)$, we get a genus $g$ surface $\Sigma$ with $2n$ boundary components $Z_1, \dots, Z_{2n}$ together with the following data:
\begin{enumerate}
\item a collection $\mb{\alpha}^\text{arc}=\{\alpha^\text{arc}_1, \dots , \alpha^\text{arc}_{2n-1}\}$ of disjoint, properly embedded, oriented arcs, with the arc $\alpha_i^\text{arc}$ connecting $Z_i$ to $Z_{i+1}$,
\item a collection of pairwise disjoint, homologically linearly independent, simple closed curves $\mb{\alpha}=\{\alpha_1, \dots, \alpha_g\}$,
\item a collection of pairwise disjoint,  simple closed curves $\mb{\beta}=\{\beta_1, \dots, \beta_{g+n-1}\}$ with the following property: the surface obtained by cutting $\Sigma$ along $\mb{\b}$ has $n$ connected components, and each of them contains exactly two of the boundary circles $Z_1, \dots, Z_{2n}$.
\end{enumerate}
We call the collection of such data an
\textit{upper Heegaard diagram}. Notice that the boundary components of an upper Heegaard diagram $\cH^{\text{up}}=(\Sigma, \mb{\alpha}, \mb{\alpha}^\text{arc}, \mb{\beta})$ come matched together: a component $Z_i$ is paired with a component $Z_j$ if they can be connected by an arc in the complement of the $\beta$-curves. In the case of knot diagrams this matching is exactly the matching $M$ of $\{1, \dots, 2n\}$ in which $\{i,j\} \in M$ iff $i$ and $j$ are the endpoints of a strand of $K_{[t, +\infty)}$.
See \Cref{fig:matching}.

We will be working with pseudo-holomorphic disks in $\Sym^{g+n-1}(\bar{\Sigma})$, where $\bar{\Sigma}$ denotes the closed Riemann surface obtained from $\Sigma$ by pinching its boundary circles. We denote by $z_1, \dots, z_{2n}$ the punctures of $\bar{\Sigma}$ corresponding to the boundary circles of $\Sigma$. Notice that an alpha arc $\alpha^\text{arc}_i$ extends to a curve $\overline{\alpha}^\text{arc}_i \subset \bar{\Sigma}$ with $\partial \overline{\alpha}^\text{arc}_i=z_{i}-z_{i+1}$.  As in the closed case, the $\alpha$- and the $\beta$-curves specify two Lagrangian submanifolds inside $\Sym^{g+n-1}(\bar{\Sigma})$. More specifically we set $L_\beta= \beta_1 \times \dots \times \beta_{g+n-1}$ and 
\[L_\alpha= \alpha_1\times \dots \times \alpha_g \times \Sym^{n-1}\left(\overline{\alpha}_{1}^\text{arc} \cup \dots \cup \overline{\alpha}_{2n-1}^\text{arc} \right) \ .\]
Notice that $L_\alpha$ is a singular space with singularities at points having more than one  coordinate on the same $\alpha$-arc. If we set $L_\alpha^*=L_\alpha \setminus L_\alpha^\text{sing}$ then we have that
\[L_\alpha^*= \bigcup_{\{i_1, \dots , i_{n-1}\}\subset \{1, \dots, 2n-1\}} \alpha_1\times \dots \times \alpha_g \times \alpha_{i_1}^\text{arc} \times \dots \times \alpha_{i_{n-1}}^\text{arc} \ . \]  

\begin{exr} Prove that the intersection points of $L_\alpha^* \cap L_\beta$ are in one to one correspondence with upper Kauffman states. Show that the local state of the upper Kauffman state associated to an intersection point $\x \in (\alpha_1\times \dots \times \alpha_g \times \alpha_{i_1}^\text{arc} \times \dots \times \alpha_{i_{n-1}}^\text{arc}) \cap L_\beta$ is given by $s(\x)= \{1, \dots, 2n-1\} \setminus \{i_1, \dots , i_{n-1}\} $ 
\end{exr} 

Let $\x$ and $\y \in L_\alpha^* \cap L_\beta$ be two intersection points. We denote by $\pi_2(\x, \y)$ the set of homotopy classes of Whitney disks connecting $\x$ to $\y$ with left side lying on $L_\alpha$ and right side on $L_\beta$.  We associate to a homotopy class $\phi \in \pi_2(\x, \y)$ an algebra element $c(\phi) \in \mathcal{C}(n)$ as follows. In proximity of the punctures $z_1, \dots, z_{2n}$ place base points $r_1, \dots, r_{2n}$ and $\ell_1, \dots, \ell_{2n}$ as suggested by Figure \ref{fig:boundary_markings}. As in the closed case a homotopy class $\phi \in \pi_2(\x, \y)$ has multiplicities $r_i(\phi)=\#| \phi(D^2) \cap V_{r_i}|$ and $\ell_i(\phi)=\#| \phi(D^2) \cap V_{\ell_i}|$. Suppose now $\x$ and $\y$ are two intersection points corresponding to Kauffman states $(x, \kappa_x)$ and $(y, \kappa_y)$ respectively. Define $c(\phi)= (x, y, U^{t_1}\dots U^{t_{2n}}) \in C(n)$, where the exponents $t_1,\dots, t_{2n}$ are chosen so that $c(\phi)$ has weight vector 
\[n(\phi)= \left(\frac{r_1(\phi)+\ell_1(\phi)}{2}, \dots, \frac{r_i(\phi)+\ell_i(\phi)}{2} , \dots , \frac{r_{2n}(\phi)+\ell_{2n}(\phi)}{2} \right) \ \  \in \Z^{2n} \ .\]

\begin{figure}
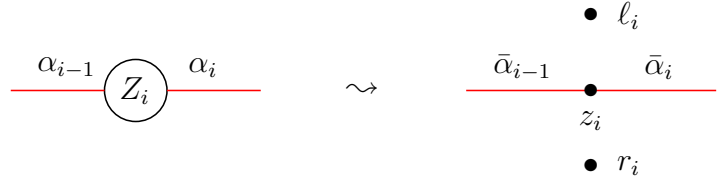

\begin{overlay}
\pict{boundary_markings.pdf}{0.62}
\toptext{$Z_i$}{-3,0}
\toptext{$z_i$}{3,-0.4}
\toptext{$\bullet$}{3,0}
\toptext{$\ell_i$}{3.5,1}
\toptext{$\bullet$}{3,1}
\toptext{$r_i$}{3.5,-1}
\toptext{$\bullet$}{3,-1}
\toptext{$\al_{i-1}$}{-3.9,0.3}
\toptext{$\al_{i}$}{-2.1,0.3}
\toptext{$\bar\al_{i}$}{3.9,0.3}
\toptext{$\bar\al_{i-1}$}{2.1,0.3}
\toptext{$\leadsto$}{0,0}
\end{overlay}
\caption{(Left) Some boundary component $Z_i$ of $\Sigma$.
(Right) Boundary component $Z_i$ filled in and replaced by the point $z_i$.
Then the points $\ell_i$, and $r_i$ are placed on opposite sides
of $\bar\al_i \un \bar \al_{i-1}$.
.}
\label{fig:boundary_markings}
\end{figure}

\begin{lem}\label{additivity} Given $\phi_1\in \pi_2(\x, \y)$ and $\phi_2\in \pi_2(\y, \bold{z})$ we have $c(\phi_1) \cdot c(\phi_2)= c(\phi_1 * \phi_2)$.
\end{lem}
\begin{proof} 
This is true provided that $c(\phi_1) \cdot c(\phi_2)$ and $c(\phi_1 * \phi_2)$ have the same weight vector. On the other hand we have that $w(c(\phi_1 * \phi_2))=n(\phi_1 * \phi_2)=n(\phi_1)+n(\phi_2)=w(c(\phi_1))+w(c(\phi_2))=w(c(\phi_1)\cdot c(\phi_2))$.
\end{proof}

We define the structure map $\dd : DFK(K_{[t, +\infty)}) \fromto \cC\left( n \right) \otimes_{I(n)} DFK(K_{[t, +\infty)})$ as follows. Given an intersection point $\x \in L_\alpha \cap L_\beta$ set  
\begin{equation*}
\dd (\mathbf{x}) = \sum_{\y} \xi_{\x ,\y} \tp \y
\end{equation*}
where 
\[ \xi_{\x ,\y}=\sum_{\phi\in \pi_2( \mathbf{x} , \mathbf{y}),\mu( \phi) = 1}
\#\left(\frac{\cM( \phi) }{\RR}\right)\cdot c( \phi )  \ . \]

\begin{thm}$\delta$ is the structure map of an $\omega_M$-curved D-structure.
\label{lem:structure_map}
\end{thm}
\begin{proof}[Sketch of the proof] We have to prove that the following relations between the structure 
constants hold:
\[ \sum_{\y} \xi_{\x,\y} \cdot \xi_{\y , \bold{z}}= \begin{cases} \ \omega_M \ \ \text{ if } \x = \bold{z} \\ \ 0 \ \ \text{ otherwise}  \end{cases}  \ .\]
Because of Lemma \ref{additivity} we have 
\[\sum_{\y} \xi_{\x,\y} \cdot \xi_{\y , \bold{z}}= \sum_{\mu(\phi)=2}\lambda(\phi) \cdot c(\phi)\]
where $\lambda(\phi) \in \Z_2$ is given by
\[\lambda(\phi)=  \sum_{\phi_1* \phi_2=\phi}  \#\left(\frac{\cM(\phi_1) }{\RR}\right) \cdot \#\left(\frac{\cM(\phi_2) }{\RR}\right) \ . \]

This leads us to consider the ends of the moduli space of holomorphic strips in homotopy classes $\phi \in \pi_2(\x, \bold{z})$ with $\mu(\phi)=2$. For the limit $u$ of a sequence of pseudo-holomorphic disks approaching the boundary of the moduli space $\mathcal{M}(\phi)/\R$ there are a few possibilities:
\begin{enumerate}
\item $u=u_1 * u_2$ is a broken flow-line with $u_1$ connecting $\x$ to an intersection point $\y \in L_\alpha^* \cap L_\beta$ and $u_2$ connecting $\y$ to $\bold{z}$ (with both $u_1$ and $u_2$ having Maslov index one)
\item $u=u_1 * u_2$ is a broken flow-line with $u_1$ connecting $\x$ to an intersection point $\y \in (L_\alpha \cap  L_\beta) \setminus (L_\alpha^* \cap  L_\beta) $ and $u_2$ connecting $\y$ to $\bold{z}$
\item $\x= \bold{z}$ and the boundary of $u$ is entirely mapped on $\overline{L_\alpha^*}$ 
\item $\x= \bold{z}$ and the boundary of $u$ is entirely mapped on $L_\beta$. 
\end{enumerate}

As a consequence of the relations (R-1) and (R-2) we imposed over $\cC(n)$ one can see that if degenerations such as $\left( 1 \right)$ and $\left( 3 \right)$ occur, then $c(\phi)=0$. Thus, one only has to consider moduli spaces of strips with $\mu(\phi)=2$ in which the other types of degenerations occur. For these moduli spaces when $\x\not=\bold{z}$ we have
\[0=\# \left| \partial \left(\frac{\cM(\phi) }{\RR}\right) \right| =   \sum_{\phi_1* \phi_2=\phi}  \#\left(\frac{\cM(\phi_1) }{\RR}\right) \cdot \#\left(\frac{\cM(\phi_2) }{\RR}\right)= \lambda(\phi)\ \ \ \text{mod }2 \ , \]
proving that $\sum_{\y} \xi_{\x,\y} \cdot \xi_{\y , \bold{z}}=0$ when $\x\not=\bold{z}$.   

On the other hand, when $\x=\bold{z}$ there are exactly $n$ homotopy classes of disks $\psi$ with boundary completely lying on $L_\beta$. These correspond to the matchings $\{i,j\}\in M$, and are represented on the upper Heegaard diagram by the domains of \Cref{fig:matching}. For these domains one computes $c(\phi)=U_i \cdot U_j$, proving that   
\[\sum_{\y} \xi_{\x,\y} \cdot \xi_{\y , \x}= \sum_{\{i,j\} \in M}U_i \cdot U_j = \omega_M \ ,\]
and we are done. 
\end{proof}

\begin{figure}
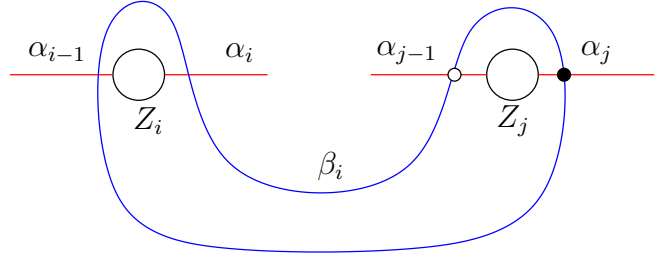

\begin{overlay}
\pict{matching.pdf}{0.57}
\toptext{$\b_i$}{0,-0.5}
\toptext{$\al_{i-1}$}{-3.6,1}
\toptext{$\al_{i}$}{-1.2,1}
\toptext{$Z_i$}{-2.4,0.1}
\toptext{$\al_{j-1}$}{1,1}
\toptext{$\al_{j}$}{3.5,1}
\toptext{$Z_j$}{2.4,0.1}
\end{overlay}
\caption{The matching $\left\{ i,j \right\}\in M$ 
as represented on the upper Heegaard diagram.}
\label{fig:matching}
\end{figure}

\begin{rmk} For sufficiently small $\epsilon >0$,  $CFK(K_{[-10+\epsilon, +\infty)})$ is a type $D$-structure over $\cC(2)=\F[U,V]/UV$. In this case $\omega_M=UV=0$, \text{i.e.} $CFK(K_{[-10+\epsilon, +\infty)})$ is flat. Obviously $CFK(K_{[-10+\epsilon, +\infty)})=CFK(K)$ as chain complexes.
\end{rmk} 

\subsection{The Gluing Theorem}
Similarly but in a somehow more complicated way, one may define the structure maps of an $\cA_\infty$-module 
\[m_j:AFK\left(K_{(-\infty,t]}\right) \otimes \cC(n)^{\otimes j-1} \to AFK\left(K_{(-\infty,t]} \right) \]
and of a $DA$-bimodule
\[\delta_{1+j}: DAFK\left(K_{[t,s]}\right)\otimes C(n)^{\otimes j}\to C(m)\otimes DAFK\left(K_{[t,s]}\right)  \]
where $n$ denotes the number of endpoints at the $z=t$ level and $m$ the number of endpoints at $z=t'$. We will not explicitly define these maps here. The main point of this construction is that one can compute the differential of
\[CFK(K)= AFK(K_{(-\infty, t]}) \otimes_{I(n)} DFK(K_{[t, +\infty)}) \ ,\] 
and more generally the structure map of 
\[DFK(K_{[t_1, +\infty)})= DAFK(K_{[t_1, t_2]}) \otimes_{I(n_1)} DFK(K_{[t_1, +\infty)})\]   
by means of Equation \eqref{eqn:d_structure_map_on_tensor_product}. More precisely we have the following theorem, whose proof relies on some gluing results about $J$-holomorphic curves.

\begin{thm}[Gluing Theorem, Ozsv\'ath-Szab\' o]For $-10<t_1<t_2<+10$ we have an identification of (curved) type $D$-structures
\[DFK(K_{[t_1, +\infty)})= DAFK(K_{[t_1, t_2]}) \boxtimes DFK(K_{[t_2, +\infty)}) \ .\]
In particular for $t \in (-10, +10)$ one has
\[CFK(K)= AFK(K_{(-\infty, t]}) \boxtimes DFK(K_{[t, +\infty)}) \ .\] 

\vspace{-1.2\lineheight}
\QEDB 
\end{thm}

The details of these analytic constructions, as well as a proof of the Gluing Theorem will appear in a forthcoming work of Ozsv\' ath and Szab\' o.

\section{The bimodules of some elementary configurations}\label{bimodules}
Any knot $K \subset \R^3$ can be presented by a diagram obtained by concatenating elementary configurations as in \Cref{fig:maxima,fig:parallel,fig:crossings,fig:local_min}.  This can be done by suitably modifying a knot projection in standard position as in \Cref{fig:alter} below. 
We now explicitly describe the structure maps of the curved $DA$-bimodules associated to these partial diagrams. Based on the Gluing Theorem one can then use the maps coming from these elementary computations to express the differential of knot Floer homology. This leads to a combinatorial formulation of knot Floer homology (see Section \cref{sec:final} below) alternative to grid homology \cite{grid}.

\subsection{Maxima} While scanning a knot projection in standard position all maxima appear at the same level, say $z=+10$. The portion of the diagram in between the slice $z=10$ and $z=10-\epsilon$ will then look like Figure \ref{fig:maxima}. In this case there is only one upper Kauffman state, corresponding to the local state $x$ taking the intervals: $(1,2), (3,4), \dots, (2n-1, 2n)$. See Figure \ref{fig:maxima}. Idempotents act on $x$ in the obvious way: $I_y \cdot x=x$ if $x=y$, zero otherwise. Furthermore, one computes $\delta(x)=0$.    
As such, the module is isomorphic to $\ZZ_2$.

\subsection{No critical points, no gain} To make a sanity check, let us consider the case of the trivial diagram in which $2n$ strands run parallel to each other as in Figure \ref{fig:parallel}. In this case $DAFK$ is identified with the ring of the idempotents $I(n)\subset \cC(n)$ with the left-multiplication action. The structure maps $\delta_{1+j}:I(n) \otimes \cC(n)^{\otimes j} \to \cC(n)$  are all zero except for $\delta_2: I(n) \otimes \cC(n) \to \cC(n)$. For the latter we have
$\delta_2(I_x \otimes (x,y, U_1^{t_1} \dots U_{2n}^{t_{2n}}))= (x,y, U_1^{t_1} \dots U_{2n}^{t_{2n}}) \otimes I_y$,
and zero otherwise. 

\begin{exr}
Let $(X, \delta_X)$ be the type $D$-structure associated to an upper knot diagram with $2n$ endpoints. Prove that $(I(n), \delta_{1+*}) \boxtimes (X,\delta_X) = (X, \delta_X) $. (Hint: 
\[\delta_{I(n)\boxtimes  X}= \sum_{j=0}^{+\infty} (\delta_{1+j} \otimes id_X) \circ (id_{I(n)} \otimes \delta_X^j)= (\delta_{2} \otimes id_X) \circ (id_{I(n)} \otimes \delta_X),\]
so one has just to prove that $(\delta_{2} \otimes id_X) \circ (id_{I(n)} \otimes \delta_X)(1 \otimes \x)= 1 \otimes \delta_X(\x)$ for any generator $\x=(\kappa, x)$ of $X$. Recall that in $I(n)$ we have $1= \sum_{\text{local states}} I_x$.)
\end{exr}

\subsection{Local Minima} 
Suppose now that while scanning a knot $K$ we encounter a local minimum configuration between two ``consecutive'' levels $z=t_1$ and $z=t_2$. See Figure \ref{fig:local_min}. The generators of $\mathcal{M}_{n+1}= DAFK(K_{[t_1, t_2]})$ of such a configuration look as in Figure \ref{fig:local_min}. We now describe the $\delta_{1+j}$ maps.

Denote by $M$ the matching of the upper portion of the diagram $K_{[t_2, + \infty)}$. Since $K$ is a knot the first two endpoints of $K_{[t_2, + \infty)}$ are not matched together. Let $s$ and $t \in \{2, \dots, 2n+2\}$ be the endpoints matched with $1$ and $2$ according to $M$. Notice that while passing to the level $z=t_1$ we get a new matching $M'$ in which $s$ and $t$ are matched together, and $\{i,j\} \in M'$ iff $\{i,j\}\in M\setminus \{\{1,s\}, \{2,t\}\}$.

Let $\x=(\emptyset,\x^{in}, \x^{out})$ and $\y=(\emptyset,\y^{in}, \y^{out})$ be generators of $DAFK$, and $a_1=(a_1^{in}, a_1^{out}, m_1(U_{1}, \dots,U_{2n+2})),\dots ,a_j=(a_j^{in}, a_j^{out}, m_j(U_{1}, \dots,U_{2n+2}))$ denote algebra elements in $\cC(n+1)$. Write
\[\delta_{j+1} (\x \otimes a_1 \otimes \dots \otimes a_j)= \sum_{\y} \xi_{\x, \y}(a_1, \dots, a_j) \cdot \y  \]
with $\xi_{\x, \y}(a_1, \dots, a_j) \in \cC(n)$. The coefficients $\xi_{\x, \y}(a_1, \dots, a_j)$ are usually zero unless  $j=2m+1$ is odd,  
\[\x^{in}=a_1^{in}, \ \ a_1^{out}=a_2^{in},\  \  \dots , \ \ a_{2m}^{out}=a_{2m+1}^{in}, \ \ a_{2m+1}^{out}=\y^{in} \ , \]
and the weight vectors of the $a_i$'s are of the form 
\begin{align*}
&w(a_{2\ell-1})= \left(0, \frac{1}{2}+ r_i, \dots  \right) \ \ \ \  \ell=1, \dots, m \\ 
&w(a_{2\ell})= \left(1+p_\ell, 0 , \dots  \right) \ \ \ \  \ \ \ \ \  \ell=1, \dots, m  \ \ .
\end{align*}
In this case we take $\xi_{\x, \y}(a_1, \dots , a_{2m+1})=(\x^{out}, \y^{out}, U_1^{t_1}\dots  U_{2n}^{t_2n})$, where the exponents $t_1, \dots , t_{2m} \in \Z_{\geq 0}$ are chosen so that $\xi_{\x, \y}(a_1, \dots , a_{2m+1})$ has weight vector $w=(w_1, \dots, w_{2n})$ with components: 
\begin{align*}
&w_i= \sum_{k} w_{i+2}(a_{k})  \ \ \ \  \ \ \ \ \ \ \ \ \ \ \ \ \ \ \ \  \text{ if } i\not=t-2, \  s-2\\
&w_{t-2}= \sum_{k} w_{i+2}(a_{k}) + \sum_\ell p_\ell\\
&w_{s-2}= \sum_{k} w_{i+2}(a_{k}) + \sum_\ell r_\ell \ .
\end{align*} 
In particular the map $\delta_{1+j}:\mathcal{M}_{n+1} \to \cC(n+1)^{\otimes j} \to \cC(n)$ vanishes for $j$ even. 

\begin{figure}
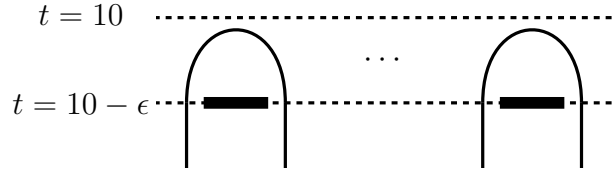

\begin{overlay}
\pict{maxima.pdf}{0.4}
\toptext{$t = 10$}{-4,1}
\toptext{$t = 10 - \e$}{-4,-0.2}
\toptext{$\cdots$}{0,0.4}
\end{overlay}
\caption{All of the maxima are on the same level between $z = 10$ and $z = 10-\e$.
Any knot in standard position can be suitably modified to look like this. 
Note that there is only one upper Kauffman state, which is the one pictured here
where we use a dot to denote an occupied region.}
\label{fig:maxima}
\end{figure}

\begin{figure}
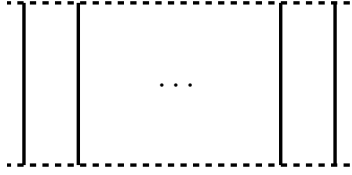

\begin{overlay}
\pict{parallel.pdf}{0.3}
\toptext{$\cdots$}{0,0}
\end{overlay}
\caption{The trivial diagram with $2n$ strands running parallel.}
\label{fig:parallel}
\end{figure}

\begin{figure}
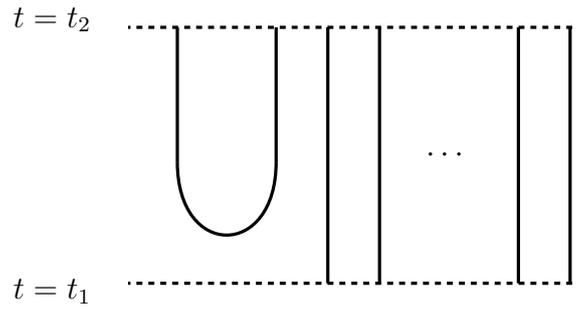

\begin{overlay}
\pict{local_min.pdf}{0.4}
\toptext{$\cdots$}{1.2,0}
\toptext{$t = t_2$}{-4,1.8}
\toptext{$t = t_1$}{-4,-1.8}
\end{overlay}
\caption{A local minimum between two levels $z = t_1$
and $z= t_2$. Any local minimum can be made to look like this one
as in \Cref{fig:alter}.}
\label{fig:local_min}
\end{figure}

\begin{figure}
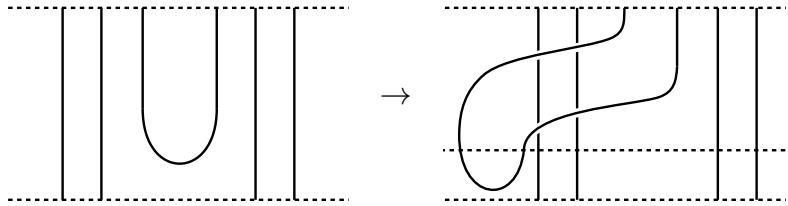

\begin{overlay}
\pict{alter1.pdf}{0.3}
\end{overlay}
$\fromto$
\begin{overlay}
\pict{alter2.pdf}{0.3}
\end{overlay}
\caption{For any local minimum in the slice $K_{\left[ t_1 , t_2 \right]}$,
we can manipulate the knot such that it instead has a local minimum of the form
in \Cref{fig:local_min}. This of course comes at the cost of additional crossings.}
\label{fig:alter}
\end{figure}

\begin{figure}
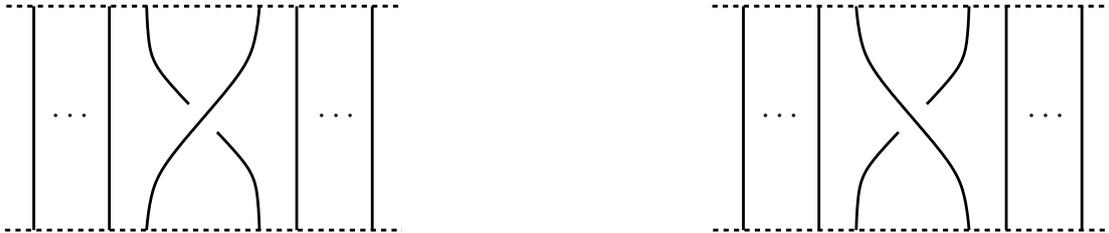

\begin{overlay}
\pict{positive_crossing.pdf}{0.35}
\toptext{$\cdots$}{-1.7,0}
\toptext{$\cdots$}{1.8,0}
\end{overlay}
\hfill
\begin{overlay}
\pict{negative_crossing.pdf}{0.35}
\toptext{$\cdots$}{-1.7,0}
\toptext{$\cdots$}{1.8,0}
\end{overlay}
\caption{(Left) The portion of the knot $K_{\left[ t_1 , t_2 \right]}$
contains a positive crossing. (Right) The portion of the knot $K_{\left[ t_1 , t_2 \right]}$
contains a negative crossing.}
\label{fig:crossings}
\end{figure}

\subsection{Crossings}\label{sec:crossings}
We now define the $DA$-bimodules associated to the crossings.
We will deal with the case of positive crossings first, and 
the case of a negative crossing will just be a slight variation of 
the same formulation.
\subsubsection{Positive crossings} 
Positive crossings are distinguished from negative crossings as in \Cref{fig:crossings}.
We construct a $DA$-bimodule $DAFK = P$ over $\cC\left( n \right)$.
This has four types of generators, which we write as
$N$, $S$, $W$, and $E$ corresponding to the possibilities near the crossing
as depicted in \Cref{fig:four_crossings}. 
These are distinguished from one another by their partial Kauffman states.
We now specify $\dd_1$, $\dd_2$, and $\dd_3$, the higher $\dd_i$'s vanish. 

\begin{figure}[t]
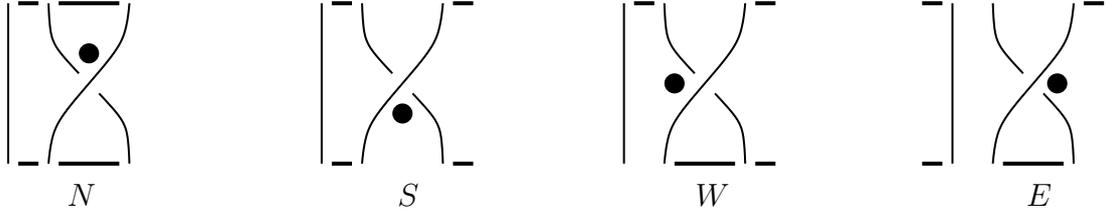

\centering
\begin{overlay}
\pict{four_crossings.pdf}{1}
\toptext{$N$}{-6.3,-1.5}
\toptext{$S$}{-2,-1.5}
\toptext{$W$}{ 2,-1.5}
\toptext{$E$}{ 6.3,-1.5}
\end{overlay}
\caption{The partial Kauffman states associated to the four types 
of positive crossings which generate the $DA$-bimodule.}
\label{fig:four_crossings}
\end{figure}

We first explicitly write:
\[\dd_1\left( N \right) = 0, \ \ \dd_1\left( E \right) = R_2 \tp  S, \ \   \dd_1\left( S \right) = 0 \ \text{ and } \dd_1\left( W \right) = L_1 \tp S \ . \]
The structure map $\dd_2$ is specified according to the following:
\begin{align*}
&\dd_2\left( S \right) = 1\tp 1 \tp S 
\\
&\dd_2\left( E \right) =
1\tp 1 \tp E
+ 1 \tp L_2 \tp N + R_2 R_1 \tp R_2 U_2 \tp N 
+ R_2 R_2 \tp U_2 \tp W
\\
&\dd_2\left( W \right) =
1 \tp 1 \tp E
+ 1 \tp R_1 \tp N + L_1L_2 \tp L_2U_1 \tp N
+ L_1L_2 \tp U_1 \tp E \\
&\dd_2\left( N \right) =
U_1 \tp R_2 \tp E + L_1L_2 \tp L_1 \tp E
+ U_2 \tp L_1 \tp W + R_2 R_1 \tp R_2 \tp W
\\
& \ \ \ \ \ \  \ \ \ + 1\tp 1 \tp N + L_1 L_2 \tp L_1 L_2 \tp N + R_2 R_1 \tp R_2 R_1 \tp N
\end{align*}
These can be extend using some local extension rules:
\begin{itemize}
\item If $b\tp Y$ appears with non-zero coefficient in
$\dd_2\left( N , a \right)$, then $\left( bU_2 \right) \tp Y$
appears with non-zero coefficient in $\dd_2\left( N , aU_1 \right)$
and $\left( bU_1 \right)\tp Y$ appears with non-zero coefficient in
$\dd_2\left( N , a U_2 \right)$.
\item If $b\tp Y$ appears with non-zero coefficient in $\dd_2\left( W , a \right)$,
then $\left( U_2 b \right)\tp Y$ appears with non-zero coefficient in
$\dd_2\left( W , U_1 a \right)$.
\item If $b\tp Y$ appears with non-zero coefficient in
$\dd_2\left( E , a \right)$, then $\left( U_1 b \right)\tp Y$ appears
with non-zero coefficient in $\dd_2\left( E , U_2 a \right)$.
\end{itemize}
for example, $\dd_2\left( W , U_1 \right) = U_2 \tp W + L_1 L_2 \tp E$.

Finally we specify $\dd_3$. We will call an algebra element \emph{elementary} if it is of the form
$p\cdot e$ for $p$ a monomial in $U_1$ and $U_2$, and
\begin{equation*}
e\in \left\{ 1 , L_1 , R_1 , L_2 , R_2 , L_1L_2 , R_2 R_1 \right\}
\end{equation*}
So let $a_1$ and $a_2$ be elementary algebra elements
such that $a_1 \tp a_2 \neq 0$. Also suppose that $U_1 U_2$ does not divide either
$a_1$ nor $a_2$. Then $\dd_3\left( S , a_1 , a_2 \right)$ is the sum of the following terms:
\begin{itemize}
\item $R_1 U_1^t \tp E$ if $\left( a_1 , a_2 \right) = \left( R_1 , R_2 U_2^t \right)$ and
$t\geq 0$.
\item $L_2 U_1^tU_2^n \tp E$ if $\left( a_1 , a_2 \right)\in$
\begin{align*}
\left\{ \left( U_1^{n+1} , U_2^t \right) , \left( R_1 U_1^n , L_1 U_2^t \right) , 
\left( L_2U_1^{n+1} , R_2  U_2^{t-1} \right)\right\} 
&& \text{when } 0\leq n < t \\
\left\{ \left( U_2^t , U_1^{n+1} \right) , 
\left( R_1 U_2^t , L_1 U_1^n \right) , \left( L_2 U_2^{t-1} , R_2 U_1^{n+1} \right)\right\}
&& \text{when } 1\leq t \leq n
\end{align*}
\item $L_2 U_2^n \tp W$ if $\left( a_1 , a_2 \right)= \left( L_2 , L_1 U_1^n \right)$ and $n\geq 0$.
\item $R_1 U_1^t U_2^n \tp W$ if $\left( a_1 , a_2 \right) \in$
\begin{align*}
\left\{ \left( U_1^{t+1} , U_1^n \right) , \left( L_2 U_2^t , R_2 U_1^n \right) , 
\left( R_1 U_2^{t+1} , L_1 U_1^{n-1} \right)\right\} && \text{when } 0\leq t < n \\
\left\{ \left( U_1^n , U_2^{t+1} \right) , \left( L_2 U_1^n , R_2 U_2^t \right) , 
\left( R_1 U_1^{n-1} , L_1 U_2^{t+1} \right)\right\} && \text{when } 1\leq n \leq t
\end{align*}
\item $L_2 U_1^t U_2^n \tp N$ if $\left( a_1 , a_2 \right) \in$
\begin{align*}
\left\{ \left( U_1^{n+1} , L_2 U_2^t \right) , \left( R_1 U_1^n , L_1L_2 U_2^t \right) , 
\left( R_1 U_2^{t+1} , U_1^n \right)\right\} && \text{when } 0\leq n < t \\
\left\{ \left( L_2 U_2^t , U_1^{n+1} \right) , \left( U_2^t , L_2 U_1^{n+1} \right)
\left( R_1U_2^t , L_1 L_2 U_1^n \right)\right\} && \text{when } 1\leq t \leq n \\
\left\{ \left( L_2 , U_1^{n+1} \right) \right\} && \text{when } 0 = t\leq n
\end{align*}
\item $R_1 U_1^t U_2^n \tp N$ if $\left( a_1 , a_2 \right)$ is in the following list:
\begin{align*}
\left\{ \left( U_2^{t+1} , R_1 U_1^n \right) , \left( L_2 U_2^t , R_2 R_1 U_1^n \right) , 
\left( R_1 U_2^{t+1} , U_1^n \right)\right\} && \text{when } 0\leq t < n \\
\left\{ \left( R_1 U_1^n , U_2^{t+1} \right) , \left( U_1^n , R_1 U_2^{t+1} \right),
\left( L_2U_1^n , R_2 R_1 U_2^t \right)\right\} && \text{when } 1\leq n\leq t \\
\left\{ \left( R_1 , U_2^{t+1} \right) \right\} && \text{when } 0 = n\leq t
\end{align*}
\end{itemize}

\subsubsection{Negative crossings}
We now construct a $DA$-bimodule $N$ over $\cC\left( n \right)$
to be associated to a negative crossing. 
These are distinguished from positive crossings as in \Cref{fig:crossings}.
This bimodule has the same four types of generators $N$,$S$,$E$, and $W$
and is just a tweaked version of $P$.
Define a map $o: \cC\left( n \right) \fromto \cC\left( n \right)$ to satisfy
$o\left( a\cdot b \right) = o\left( b \right) \cdot o\left( a \right)$ and 
in particular take:
\begin{align}
o\left( I_x \right) = I_x &&
o\left( U_t \right) = U_t \\
o\left( L_t \right) = R_t &&
o\left( R_t \right) = L_t  \ .
\end{align}
If $\dd^1_1\left( X \right) = b\tp Y$ in  $P$, 
then $\dd_1\left( Y \right) = o\left( b \right) \tp X$ in $N$. 
If $\dd_2\left( X , a \right) = o\left( b \right) \tp Y$ in $P$, 
then $\dd_2\left( Y , o\left( a \right) \right) = o\left( b \right) \tp X$ in $N$.
Finally, if $\dd_3\left( X , a_1 , a_2 \right) = b\tp Y$ in $P$,
then  $\dd_3\left( Y , o\left( a_2 \right) , o\left( a_1 \right) \right) = 
o\left( b \right) \tp X$.

\subsection{A final remark}\label{sec:final} In principle one can forget about the analytic theory and define knot Floer homology as follows. Start from a knot projection $D$ in standard position as in the beginning of \Cref{sec:upperkauffman}. Up to isotopy we can suppose that the crossings $c_1, \dots, c_m$ of $D$ all have different $z$-coordinate. Let $\overline{\mathcal{M}}$ be the type $D$-structure of Section \ref{typeD},  $\underline{\mathcal{M}}=\mathcal{M}_{n}\boxtimes \mathcal{M}_{n-1} \boxtimes \dots \boxtimes \mathcal{M}_{2}$ be the $DA$-bimodule obtained by taking the box tensor product of the $DA$-bimodules of Section \ref{bimodules}, and let $\mathcal{\cX}_i$ denote the $DA$-bimodule associated to $K_{[z(c_i)-\epsilon, z(c_i)+\epsilon]}$  as in Section \ref{sec:crossings} above. Define 
\[C(D)= \overline{\mathcal{M}}\boxtimes \mathcal{X}_1 \boxtimes \dots \boxtimes \mathcal{X}_m \boxtimes \underline{\mathcal{M}} \ .\]         

\vspace{0.1cm}
Obviously the definition of $C(D)$ only makes reference to the diagram $D$, no choice of a generic almost-complex structure needs to be done, and no holomorphic disk count is involved. In fact, in \cite{bordered1} Ozsv\' ath and Szab\' o proved that the chain homotopy type of $C(D)$ does not change if $D$ is changed by any of the Reidemeister moves. This was done locally at level of $DA$-bimodules, by proving that the braid relations are satisfied.  
\begin{thm}[Ozsv\' ath-Szab\' o]
The chain homotopy type of $C\left( D \right)$ is a knot invariant.
\end{thm}

Notice that the chain complexes coming from this construction are more ``economic'' than the one coming from the combinatorial formulation of \cite{grid}.      

\bibliography{Final}
\end{document}